\documentclass{article}

\usepackage{mdwlist}
\usepackage{amsfonts,amsmath,amsthm,amssymb}
\usepackage{epsfig}
\usepackage{enumerate}
\usepackage{a4wide}
\usepackage{makeidx}
\usepackage{color}
\usepackage{psfrag}

\newtheorem{thm}{Theorem}[section]

\newtheorem{lem}[thm]{Lemma}

\title{Degree distribution in random planar graphs}
\author{Michael Drmota\footnote{Technische Univerisit\"{a}t Wien, Institute of
Discrete Mathematics and Geometry, Wiedner Hauptstrasse 8–10,
A–1040 Wien, Austria. {\tt michael.drmota@tuwien.ac.at}}
 \and Omer Gim\'{e}nez\footnote{Universitat Polit\`{e}cnica de Catalunya, Departament de Llenguatges i Sistemes Inform\`{a}tics,
 Jordi Girona 1--3, 08034
Barcelona, Spain. {\tt omer.gimenez@upc.edu}}
  \and  Marc Noy\footnote{Universitat
Polit\`{e}cnica de Catalunya, Departament de Matem\`{a}tica Aplicada II,
Jordi Girona 1--3, 08034 Barcelona, Spain. {\tt  marc.noy@upc.edu}
 \newline
 Research supported in part by Ministerio de Ciencia e Innovaci\'{o}n
 MTM2008-03020. }}
\date{}

\begin{document}

\maketitle
\begin{abstract}
We  prove that  for each $k\ge0$, the probability that a root
vertex in a random planar graph has degree $k$ tends to a
computable constant $d_k$, so that the expected number of vertices
of degree $k$ is asymptotically $d_k n$, and  moreover that
$\sum_k d_k =1$.
The proof uses the tools developed by Gim\'{e}nez and Noy in their
solution to the problem of the asymptotic enumeration of planar
graphs, and is based on a detailed analysis of the generating
functions involved in counting planar graphs. However, in order to
keep track of the degree of the root, new technical difficulties
arise. We obtain explicit, although quite involved expressions,
for the coefficients in the singular expansions of the generating
functions of  interest, which allow us to use transfer theorems in
order to get an explicit expression for the probability generating
function $p(w)=\sum_k d_k w^k$. From this we can compute the $d_k$
to any degree of accuracy, and derive the asymptotic estimate $d_k
\sim c\cdot k^{-1/2} q^k$ for large values of $k$, where $q
\approx 0.67$ is a constant defined analytically.
\end{abstract}

\section{Introduction}

In this paper all graphs are simple and labelled with labels
$\{1,2,\dots,n\}$. As usual, a graph is planar if it can be
embedded in the plane without edge crossings. A planar graph
together with a particular embedding in the plane is called a map.
There is a rich theory of counting maps, and part of it is needed
later. However, in this paper we consider planar graphs as
combinatorial objects, regardless of how many non-equivalent
topological embeddings they may have.

Random planar graphs were introduced by Denise, Wasconcellos and
Welsh~\cite{dvw}, and since then they have been widely studied.
Let us first recall the probability model. Let $ \mathcal{G}_n$ be
the family of (labelled) planar graphs with $n$ vertices. A random
planar graph $ \mathcal{R}_n$ is a graph drawn from $
\mathcal{G}_n$ with the uniform distribution, that is, all planar
graphs with $n$ vertices have the same probability of being
chosen. As opposed to the classical Erd\H{o}s-R\'{e}nyi model, we
cannot produce a random planar graph by drawing edges
independently. In fact, our analysis of random planar graphs
relies on exact and asymptotic counting.

Several natural parameters defined on $ \mathcal{R}_n$ have been
studied, starting with the number of edges, which is probably the
most basic one. Partial results where obtained in
\cite{dvw,gm,opt,upper2}, until  it was shown by Gim\'{e}nez and Noy
\cite{gn} that the number of edges in random planar graphs obeys
asymptotically a normal limit law with linear expectation and
variance. The expectation is asymptotically $\kappa n$, where
$\kappa \approx 2.21326$ is a well-defined analytic constant. This
implies that the average degree of the vertices is $2 \kappa
\approx 4.42652$.

McDiarmid, Steger and Welsh~\cite{MSW}  showed that with high
probability a planar graph has a linear number of vertices of
degree $k$, for each $k\ge1$. Our main result is  that for each
$k\ge1$, the expected number of vertices of degree $k$ is
asymptotically $d_k n$, for computable constants $d_k \ge 0$. This
is equivalent to saying that the probability that a fixed vertex,
say vertex 1, has degree $k$ tends to a limit $d_k$ as $n$ goes to
infinity. In Theorem \ref{th:planar} we show that this limit
exists and we give an explicit expression for the probability
generating function
$$
    p(w) = \sum_{k \ge 1} d_k w^k,
    $$
from which the coefficients $d_k$ can be computed to any degree of
accuracy. Moreover, we show that $p(w)$ is indeed a probability
generating function, that is, $\sum d_k=1$.

The proof is based on a detailed analysis of the generating
functions involved in counting planar graphs, as developed in
\cite{gn}, where the long standing problem of estimating the
number of planar graphs was solved. However, in this case we need
to keep track of the degree of a root vertex, and this makes the
analysis considerably more difficult.

Here is a sketch of the paper. We start with  some preliminaries,
including the fact that for the degree distribution  it is enough
to consider connected planar graphs, and that $d_0=0$. Then we
obtain the degree distribution in simpler families of planar
graphs: outerplanar graphs (Section \ref{sec:outer}) and
series-parallel graphs (Section \ref{sec:sp}). We recall that a
graph is series-parallel if it does not contain the complete graph
$K_4$ as a minor; equivalently, if it does not contain a
subdivision of $K_4$. Since both $K_5$ and $K_{3,3}$ contain a
subdivision of $K_4$, by Kuratowski's theorem a series-parallel
graph is planar. An outerplanar graph is a planar graph that can
be embedded in the plane so that all vertices are incident to the
outer face. They are characterized as those graphs not containing
a minor isomorphic to (or a subdivision of) either $K_4$ or
$K_{2,3}$. These results are interesting on their own and pave the
way to the more complex analysis of general planar graphs. We
remark that the degree distribution in these simpler cases has
been obtained independently  in \cite{angelika,angelika2} using
different techniques.

In Section \ref{sec:maps} we compute the generating function of
3-connected maps taking into account the degree of the root, which
is an essential piece in proving the main result. We rely on a
classical bijection between rooted maps and rooted
quadrangulations \cite{BT,MS}, and again the main difficulty is to
keep track of the root degree.

The task is completed in Section \ref{sec:planar}, which contains
the analysis for planar graphs. First we have to obtain a closed
form for the generating function $B^\bullet(x,y,w)$ of rooted
2-connected planar graphs, where $x$ marks vertices, $y$ edges,
and $w$ the degree of the root: the main problem we encounter here
is solving a differential equation involving algebraic functions
and other functions defined implicitly. The second step is to
obtain singular expansions of the various generating functions
near their dominant singularities; this is particularly demanding,
as the coefficients of the singular expansions are rather complex
expressions. Finally, using a technical lemma on singularity
analysis and composition of singular expansions, we are able to
work out the asymptotics for the generating function
$C^\bullet(x,y,w)$ of rooted connected planar graphs, and from
this the probability generating function can be computed exactly.
We also compute the degree distribution for 3-connected and
2-connected planar graphs. Finally in Section \ref{sec:density} we
show that there exists a computable degree distribution for planar
graphs with a given edge density or, equivalently, given average
degree.

For each of the three families studied we obtain an explicit
expression, of increasing complexity, for the probability
generating function $p(w) = \sum_{k \ge 1} d_k w^k$ . Theorems
\ref{th:outer}, \ref{th:sp} and \ref{th:planar} give the exact
expressions in each case. We remark that the expression we obtain
for $p(w)$ in the planar case is quite involved and needs several
pages to write it down. However, the functions involved are
elementary and computations can be performed with the help of
\textsc{Maple}.

The following table shows the approximate values of the
probabilities $d_k$ for small values of $k$, which are obtained by
extracting coefficients in the power series $p(w)$, and can be
computed to any degree of accuracy.

\begin{table}[htb]
$$
\renewcommand{\arraystretch}{1.2}
\begin{tabular}{|l|c|c|c|c|c|c|}
\hline
            & $d_1$ & $d_2$ & $d_3$ & $d_4$ &
            $d_5$& $d_6 $ \\
            \hline
Outerplanar     & 0.1365937 & 0.2875331 & 0.2428739 & 0.1550795 & 0.0874382 & 0.0460030 \\
Series-Parallel & 0.1102133 & 0.3563715 & 0.2233570 & 0.1257639 & 0.0717254 & 0.0421514 \\
Planar          & 0.0367284 & 0.1625794 & 0.2354360 & 0.1867737 & 0.1295023 & 0.0861805 \\
Planar 2-connected
                & 0         & 0.1728434 & 0.2481213 & 0.1925340 & 0.1325252 & 0.0879779 \\  
Planar 3-connected
                & 0         & 0         & 0.3274859 & 0.2432187 & 0.1594160 & 0.1010441 \\  
 \hline
\end{tabular}
$$
\caption{Degree distribution for small degrees.}\label{taula-deg}
\end{table}

We also determine the asymptotic behaviour for large $k$, and the
result we obtain in each case is a geometric distribution modified
by a suitable subexponential term. We perform the analysis for
connected and 2-connected graphs, and also for 3-connected graphs
in the planar case.
Table \ref{taula-est} contains a summary of the main results from
sections \ref{sec:outer}, \ref{sec:sp} and \ref{sec:planar}. It is
worth noticing that the shape of the asymptotic estimates for
planar graphs agrees with the general pattern for the degree
distribution in several classes of maps, where maps are counted
according to the number of edges~\cite{liskovets}.

\begin{table}[htb]
$$
\renewcommand{\arraystretch}{1.5}
\begin{tabular}{|l|l|l|l|l|l|l|}
\hline & connected & $q$ & 2-conn. &$q$ & 3-conn. & $q$
\\
\hline Outerplanar   &   $ c\cdot k^{1/4} e^{c'\sqrt{k}} q^k $  &
$ 0.3808138 $  & $c \cdot  k\, q^k$ & $\sqrt 2-1$
& & \\
Series-Parallel  & $ c\cdot k^{-3/2} q^k $   &   $  0.7504161 $ &
$ c\cdot k^{-3/2} q^k $ & $ 0.7620402 $
& & \\
Planar           & $c \cdot k^{-1/2} q^k$  & $0.6734506 $ &
 $c \cdot k^{-1/2} q^k$ & $  0.6734506$
&  $ c\cdot k^{-1/2} q^k $   &   $ \sqrt 7 -2  $ \\
\hline
\end{tabular}
$$
\caption{Asymptotic estimates of $d_k$ for large $k$. The
constants $c$ resp.\ $c'$ and $q$ in each case are defined
analytically. The two approximate values in the last row are
exactly the same constant.}\label{taula-est}
\end{table}

As a final remark, let us mention that in a companion paper
\cite{companion}, we prove a central limit theorem for the number
of vertices of degree $k$ in outerplanar and series-parallel
graphs, together with strong concentration results. It remains an
open problem to show that this is also the case for planar graphs.
Our results in the present paper show that the degree distribution
exists and moreover can be computed explicitly.

\section{Preliminaries}\label{sec:prelim}

For background on generating functions associated to planar
graphs, we refer to \cite{gn} and \cite{SP}, and to \cite{noy} for
a less technical description. For background on singularity
analysis of generating functions, we refer to the \cite{FO} and to
the forthcoming book by Flajolet and Sedgewick \cite{FS}.

For each class of graphs under consideration, $c_n$ and $b_n$
denote, respectively, the number of connected and 2-connected
graphs on $n$ vertices. For the three graphs classes under
consideration, outerplanar, series-parallel, and planar, we have
both for $c_n$ and $b_n$ estimates of the form
\begin{equation}\label{asymp-shape}
    c \cdot n^{-\alpha} \rho^{-n} n!,
\end{equation}
where $c, \alpha$ and $\rho$ are suitable constants \cite{SP,gn}.
For outerplanar and series-parallel graphs we have $\alpha =
-5/2$, whereas for planar graphs $\alpha = -7/2$. A general
methodology for graph enumeration explaining these critical
exponents has been developed in \cite{3-conn}.

We introduce the  exponential generating functions $C(x) = \sum
c_n x^n/n!$ and $B(x) = \sum b_n x^n/n!$. Let $C_k$ be the
exponential generating function (GF for short) for rooted
connected graphs, where the root bears no label and has degree
$k$; that is, the coefficient $[x^n/n!]C_k(x)$ equals the number
of rooted connected graphs with $n+1$ vertices, in which the root
has no label and has degree $k$. Analogously we define $B_k$ for
2-connected graphs. Also, let
$$
    B^\bullet(x,w) = \sum_{k\ge 2} B_k(x) w^k, \qquad
C^\bullet(x,w) = \sum_{k\ge 0} C_k(x)   w^k.
$$

A basic property shared by the classes of outerplanar,
series-parallel and planar graphs is that a connected graph $G$ is
in  the class if and only if the 2-connected components of $G$ are
also in the class. As shown in \cite{gn}, this implies the basic
equation
 $$C'(x) = e^{B'(xC'(x))}$$
  between univariate GFs. If we introduce the degree of the root,
  then the equation  becomes
\begin{equation}\label{basic}
    C^\bullet(x,w) = e^{B^\bullet(xC'(x),w)}.
\end{equation}
 The reason is
that only the 2-connected components containing the root vertex
contribute to its degree.

Our goal in each case is to estimate $[x^n]C_k(x)$, since the
limit probability that a given fixed vertex has degree $k$ is
equal to
\begin{equation}\label{mu}
    d_k = \lim_{n\to\infty} {[x^n]C_k(x) \over [x^n]C'(x)}.
\end{equation}
Notice that in the denominator we have the coefficient of $C'(x)$,
corresponding to vertex rooted graphs in which the root bears no
label, in agreement with the definition of $C_k(x)$.

A first observation is that the asymptotic degree distribution is
the same for connected members of a class than for all members in
the class. Let $G(x)$ be the GF for all members in the class, and
let $G_k(x)$ be the GF of all rooted graphs in the class, where
the root has degree $k$. Then we have
 $$
  G(x) = e^{C(x)}, \qquad G_k(x) = C_k(x) e^{C(x)}.
 $$
The first equation is standard, and in the second equation the
factor $C_k(x)$ corresponds to the connected component containing
the root, and the second factor to the remaining components. The
functions $G(x)$ and $C(x)$ have the same dominant singularity.
Given the singular expansions of $G(x)$ and $C(x)$ at the dominant
singularity  in each of the cases under consideration, it follows
that
$$
    \lim_{n\to\infty} {[x^n]G_k(x) \over [x^n]G'(x)} =
   \lim_{n\to\infty} {[x^n]C_k(x) \over [x^n]C'(x)}.
$$
Hence, in each case we only need to determine the degree
distribution for connected graphs. A more intuitive explanation is
that the largest component in random planar graphs eats up almost
everything: the expected number of vertices not in the largest
component is constant \cite{surfaces}.

Another observation is that  $d_0=0$ and $d_1 = \rho$, where
$\rho$ is the constant appearing in the estimate
(\ref{asymp-shape}) for $c_n$; as we are going to see, $\rho$ is
the radius of convergence of $C(x)$. Indeed, there are no vertices
of degree zero in a connected graph, and the number of connected
graphs in which the root has degree one is
 $n (n-1) c_{n-1}/ n c_n \sim \rho$.

The general approach we use for computing the $d_k$ is the
following. Let $f(x)=xC'(x)$ and let $H(z) = e^{B^\bullet(z,w)}$,
where $w$ is considered as a parameter. Let also $\rho$ be the
radius of convergence of $C(x)$, which is the same as that of
$f(x)$. According to (\ref{basic}) we have to estimate
$[x^n]H(f(x))$, and this will depend on the behaviour of $H(z)$ at
$z = f(\rho)$. In the outerplanar and series-parallel cases,
$H(z)$ turns out be analytic at $f(\rho)$, whereas in the planar
case we have a critical composition scheme, that is, the dominant
singularity of $H(z)$ is precisely $f(\rho)$. This is a
fundamental difference and we have to use different tools
accordingly. Another difference is that $ B^\bullet$ is much more
difficult to determine for planar graphs.

\smallskip

Finally we comment on an asymptotic method that we  apply several
times. Suppose that $f(z) = \sum_{n\ge 0} a_n z^n$ is the power
series representation of an analytic function and $\rho> 0$ is the
radius of convergence of $f(z)$. We say that $f(z)$ is analytic in
a $\Delta$-region if $f(z)$ can be analytically continued to a
region of the form
\begin{equation}\label{eqDelta}
\Delta = \{ z \in \mathbb{C}: |z|< \rho + \eta,\ |\arg(z-\rho)| >
\theta\},
\end{equation}
for some $\eta>0$ and $0< \theta < \frac{\pi}2$.

If we know that $|f(z)| \le C\cdot |1- z/\rho|^{-\alpha}$ for
$z\in \Delta$, then it follows that $|a_n|\le C'\cdot \rho^{-n}
n^{\alpha-1}$ for some $C'>0$ that depends on $C,\alpha, \eta$,
and $\theta$; see \cite{FO}. In particular, if we know that $f(z)$
is analytic in a $\Delta$-region and  has a local representation
of the form
\begin{equation}\label{eqlocalexp}
f(z) = A_0 + A_2 Z^2 + A_3 Z^3 + O(Z^4),
\end{equation}
where  $Z = \sqrt{1-z/\rho}$, then it follows that $|f(z) - A_0
-A_2 Z^2 - A_3 Z^3| \le C \cdot|1-z/\rho|^2$ for $z\in \Delta$. As
a consequence
\[
a_n = \frac{3A_3}{4 \sqrt \pi} \rho^{-n} n^{-5/2} + O(\rho^{-n} n^{-3}).
\]

In fact, we  focus mainly on the derivation of local expansions of
the form (\ref{eqlocalexp}). The analytic continuation to a
$\Delta$-region is usually  easy to establish. We  either have
explicit equations in known functions or implicit equations where
we can continue analytically with the help of the implicit
function theorem.

As a key example, we consider a function $y=y(z)$ that has a power
series representation at $z_0=0$ and that satisfies an analytic
functional equation $\Phi(y,z)= 0$. Suppose that we have $y(z_0) =
y_0$ (so that  $\Phi(y_0,z_0) = 0$) and $\Phi_y(y_0,z_0) \ne  0$.
Then the implicit function theorem implies that $y(z)$ can be
extended analytically to a neighbourhood of $z=z_0$. In particular
it follows that $y(z)$ cannot be singular at $z=z_0$. On the other
hand if we know that there exists $z_0$ and $y_0=y(z_0)$ with
\[
\Phi(y_0,z_0) = 0 \quad \mbox{and}\quad \Phi_y(y_0,z_0) = 0
\]
and the conditions
\[
\Phi_z(y_0,z_0) \ne 0 \quad \mbox{and}\quad \Phi_{yy}(y_0,z_0) \ne
0,
\]
then $z_0$ is a singularity of $y(z)$ and there is a local expansion
of the form
\[
y(z) = Y_0 + Y_1 Z + Y_1 Z^2 + \cdots,
\]
where $Z = \sqrt{1- z/z_0}$, $Y_0 = y_0$ and $Y_1 = -
\sqrt{2z_0\Phi_z(y_0,z_0)/ \Phi_{yy}(y_0,z_0)}$; see \cite{Drm97}.
 In our applications it is usually easy to show that
$\Phi_y(y(z),z)\ne 0$ for $|z|\le z_0$ and $z\ne z_0$. Hence, in this
situation $z=z_0$ is the only singularity on the boundary of the circle
$|z|\le z_0$ and $y(z)$ can be analytically continued to a $\Delta$-region.

\section{Outerplanar graphs}\label{sec:outer}

In this section $C(x)$ and $B(x)$ now denote the GFs of connected
and 2-connected, respectively, outerplanar graphs. We start by
recalling some results from \cite{SP}. From the equivalence
between rooted 2-connected outerplanar graphs and polygon
dissections where the vertices are labelled $1,2,\dots,n$ in
clockwise order (see Section 5 in \cite{SP} for details), we have
the explicit expression
$$
    B'(x) = {1 + 5x - \sqrt{1-6x+x^2} \over 8}.
    $$
The radius of convergence of $B(x)$ is $3-2\sqrt2$, the smallest
positive root of $1-6x+x^2=0$. The radius of convergence of $C(x)$
is $\rho = \psi(\tau)$, where $\psi(u) = u e^{-B'(u)}$, and $\tau$
is the unique positive root of $\psi'(u)=0$. Notice that $\tau$
satisfies $\tau B''(\tau)=1$. The approximate values are $\tau
\approx 0.17076$ and $\rho \approx 0.13659$. We also need the fact
that $\psi$ is the functional inverse of $xC'(x)$, so that $\tau =
\rho C'(\rho)$.

Let
\begin{equation}\label{eq:D}
    D(x) = {1+x - \sqrt{1-6x+x^2} \over 4}
\end{equation}
and let $D_k(x) = x(2D(x) -x)^{k-1}$ ($D_k$ is the ordinary GF for
polygon dissections in which the root vertex has degree $k$). Then
we have
$$
    B_k = {1 \over 2} D_k, \ k \ge2, \qquad  B_1=x.
    $$
By summing a geometric series we have an explicit expression for
$B^\bullet$, namely
\begin{equation}\label{B-outer}
    B^\bullet(x,w) = xw + \sum_{k=2}^\infty  {x \over 2}(2D(x)-x)^{k-1}w^k = xw +      {xw^2 \over 2} {2D(x)-x \over 1 -
    (2D(x)-x)w}.
\end{equation}

Our goal is to analyze $B^\bullet(x,w)$ and $ C^\bullet(x,w) =
\exp(B^\bullet(xC'(x),w))$. For this  we need the following
technical lemma.

\begin{lem}\label{sing}
Let $f(x) = \sum_{n\ge 0} a_n x^n/n!$ denote the exponential
generating function of a sequence $a_n$ of non-negative real
numbers and assume that $f(x)$ has exactly one dominating
square-root singularity at $x=\rho$ of the form
\[
f(x) = g(x) - h(x) \sqrt{1- x/\rho},
\]
where $g(x)$ and $h(x)$ are analytic at $x= \rho$ and $f(x)$ has
an analytic continuation to the region $\{ x\in \mathbb{C} : |x|<
\rho+\varepsilon\} \setminus \{ x \in \mathbb{R} : x \ge \rho\}$
for some $\varepsilon > 0$. Further, let $H(x,z)$ denote a function
that is analytic for  $|x|< \rho+\varepsilon$ and
$|z| < f(\rho) + \varepsilon$ such that   $H_z(\rho,f(\rho)) \ne 0$.
 Then the function
\[
f_H(x) = H(x,f(x))
\]
has a power series expansion $f_H(x) = \sum_{n\ge 0} b_n x^n/n!$
and the coefficients $b_n$  satisfy
\begin{equation}\label{eqlimitrelation}
\lim_{n\to\infty} \frac{b_n}{a_n} =  H_z(\rho,f(\rho)).
\end{equation}
\end{lem}

\begin{proof}
From $f(x) = g(x) - h(x) \sqrt{1- x/\rho}$ it follows from
singularity analysis \cite{FS} that the sequence $a_n$ is given
asymptotically  by
\[
\frac{a_n}{n!} \sim \frac{h(\rho)}{2\sqrt \pi} \,\rho^{-n}
n^{-3/2}.
\]
Since $H$ is analytic at $f(\rho)$, it has a Taylor series
$$
    H(x,z) = H(\rho,f(\rho)) + H_z(\rho,f(\rho))(z-f(\rho))
+ H_x(\rho,f(\rho))(x-\rho) + \cdots
$$
The function  $f_H(x)$ has also a square-root singularity at
$x=\rho$ with a singular expansion, obtained by composing the
analytic expansion of $H(x,z)$ with the singular expansion of
$f(x)$, namely
\[
f_H(x) = H(\rho,f(\rho)) -
H_z(\rho,f(\rho)) h(\rho) \sqrt{1- \frac x\rho} + O(|1-x/\rho|).
\]
Consequently, the coefficients $b_n$ can be estimated as
\[
\frac{b_n}{n!} \sim \frac{H_z(\rho,f(\rho)) h(\rho)}{2\sqrt \pi}
\rho^{-n} n^{-3/2},
\]
and  (\ref{eqlimitrelation}) follows.
\end{proof}

We are ready for obtaining the degree distribution of
two-connected outerplanar graphs and connected outerplanar graphs.
Both results have been obtained independently in
\cite{angelika,angelika2}, and our respective results agree.

\begin{thm}\label{th:twoconnectedouter}
Let $d_k$ be the limit probability that a vertex of a
two-connected outerplanar graph has degree $k$. then
 $$
 p(w) = \sum_{k \ge 1} d_k w^k = \frac{2(3-2\sqrt 2) w^2}
 {(1-(\sqrt 2 -1)w)^2} = \sum_{k\ge 2} 2(3-2\sqrt 2)(k-1) (\sqrt 2 -1)^{k}w^k.
 $$

Moreover $p(1)=1$, so that the $d_k$ are indeed a probability
distribution.
\end{thm}

\begin{proof}
Since $B^\bullet(x,1) = B'(x)$ and $D(x) = 2B'(x)-x$ we can represent
$B^\bullet(x,w)$ as
\[
    B^\bullet(x,w) =  xw +  {xw^2 \over 2} {4B'(x)-3x \over 1 -
    (4B'(x)-3x)w}.
\]
Hence, by applying Lemma~\ref{sing} with $f(x) = B'(x)$ and
\[
H(x,z) = xw +  {xw^2 \over 2} {4z-3x \over 1 -
    (4z-3x)w}
\]
we  obtain
\[
p(w) = \lim_{n\to\infty} \frac{[x^n]B^\bullet(x,w)}{[x^n]B'(x)} =
\frac{2(3-2\sqrt 2) w^2}
 {(1-(\sqrt 2 -1)w)^2}.
\]
Note that $\rho = 3 - 2 \sqrt 2$ and that $w$ is considered here
as an additional (complex) parameter.
\end{proof}

\begin{thm}\label{th:outer}
Let $d_k$ be the limit probability that a vertex of a connected
outerplanar graph has degree $k$. then
 $$
 p(w) = \sum_{k \ge 1} d_k w^k = \rho \cdot {\partial \over
 \partial x} \, e^{B^\bullet(x,w)} \left|_{x=\rho C'(\rho)}\right.,
 $$
 where $B^\bullet$ is given by Equations (\ref{eq:D}) and
 (\ref{B-outer}).

Moreover $p(1)=1$, so that the $d_k$ are indeed a probability
distribution and we have asymptotically, as $k\to\infty$
\[
d_k \sim c_1 k^{ 1/4} e^{c_2 \sqrt k} q^k,
\]
where $c_1 \approx 0.667187$, $c_2 \approx 0.947130$, and $q = 2D(\tau) - \tau \approx
0.3808138$.
\end{thm}

\begin{proof}
We have
$$
\sum_k C_k(x) w^k = C^\bullet(x,w) = e^{B^\bullet(xC'(x),w)}.
$$
The radius of convergence $3-2\sqrt2$ of $B(x)$ is larger than
$\rho C'(\rho) = \tau \approx 0.17076$.  Hence we can apply the
previous lemma with $f(x) = x C'(x)$ and $H(z) =
e^{B^\bullet(z,w)}$, where $w$ is considered as a parameter. Then
we have
$$
{\partial \over
 \partial x} e^{B^\bullet(x,w)}\left|_{x=\rho C'(\rho)}\right. =
    \lim_{n \to \infty} {[x^n] C^\bullet(x,w) \over [x^n] xC'(x)}
    = \lim_{n\to \infty} \sum_{k \ge 1} \rho^{-1}{[x^n] C_k(x) \over [x^n]
    C'(x)}w^k  =\rho^{-1} \sum_{k\ge1} d_k w^k,
$$
and the result follows.

For the second assertion let us note that $B^\bullet(x,1) =
B'(x)$. If we recall that $\rho C'(\rho) = \tau$ and $\tau
B''(\tau) = 1$, then
$$
    p(1) = \rho \, e^{B'(\tau)} B''(\tau) = \rho C'(\rho)
    \tau^{-1} = 1.
    $$

In order to get an aysmptotic expansion for $d_k$ we have to
compute $p(w)$ explicitly:
\[
p(w) = \rho\frac{\tau(2D(\tau) - \tau)(2D'(\tau)-1)w^2}{2(1-(2D(\tau) - \tau)w)^2}
\exp\left( \tau w + \frac{\tau(2D(\tau) - \tau)w^2 }
{2(1-(2D(\tau) - \tau)w)}  \right).
\]
This is a function that is admissible in the sense of Hayman \cite{Hayman}.
Hence, it follows that
\[
d_k \sim \frac{p(r_k)r_k^{-k}}{\sqrt{2\pi b(r_k)}},
\]
where $r_k$ is given by the equation $r_k p'(r_k)/p(r_k) = k$ and
$b(w) = w^2p''(w)/p(w) + wp'(w)/p(w)- (wp'(w)/p(w))^2$. A standard
calculation gives the asymptotic expansion for the coefficients
$d_k$.
\end{proof}

With the help of the explicit expression for $p(w)$ we obtain the
values for small $k$ shown in Table~\ref{taula-deg}.

\section{Series-parallel graphs}\label{sec:sp}

In this section $C(x)$ and $B(x)$ now denote the GFs of connected
and 2-connected, respectively, series-parallel graphs.
First we recall the necessary results from \cite{SP}. The radius
of convergence of $B(x)$ is $R \approx 0.1280038$. The radius of
convergence of $C(x)$ is, as for outerplanar graphs, $\rho =
\psi(\tau)$, where $\psi(u) = u e^{-B'(u)}$, and $\tau$ is the
unique positive root of $\psi'(u)=0$. Again we have that $\psi$ is
the functional inverse of $xC'(x)$, so that $\tau = \rho
C'(\rho)$, and  $\tau$ satisfies $\tau B''(\tau)=0$.
The approximate values are $\tau \approx 0.1279695$ and $\rho
\approx 0.1102133$.

In order to study 2-connected series-parallel graphs, we need to
consider series-parallel networks, as in \cite{SP}. We recall that
a network is a graph with two distinguished vertices, called
poles, such that the graph obtained by adding an edge between the
two poles is 2-connected. Let $D(x,y,w)$ be the exponential GF of
series-parallel networks, where $x,y,w$ mark, respectively,
vertices, edges, and the degree of the first pole. Define
$S(x,y,w)$ analogously for series networks. Then we have
\begin{eqnarray*}
D(x,y,w) &=& (1+yw)e^{S(x,y,w)}-1 \\
S(x,y,w) &=& \left( D(x,y,w) - S(x,y,w)\right)x D(x,y,1),
\end{eqnarray*}
The first equation reflects the fact that a network is a parallel
composition of series networks, and the second one the fact that a
series network is obtained by connecting a non-series network with
an arbitrary network (see \cite{walsh} for details); the factor
$D(x,y,1)$ appears because we only keep track of the degree of the
first pole.

\paragraph{Remark.} For the results of the present section,
we do not need to take into account the number of edges and we
could set $y=1$ everywhere. However, in the case of planar graphs
we do need the GF according to all three variables and it is
convenient to present already here the full development. In the
proof of the main result of this section, Theorem \ref{th:sp}, we
just set $y=1$.

\medskip
Set $E(x,y) = D(x,y,1)$, the GF for series-parallel networks
without marking the degree of the root, which satisfies (see
\cite{SP}) the equation
\begin{equation}\label{eq:E-sp}
 \log\left({1+E(x,y) \over 1+y}  \right)
= {xE(x,y)^2 \over 1+xE(x,y)}.
\end{equation}
From the previous equations it follows that
\begin{equation}\label{eq:D-sp}
\log\left({1+D(x,y,w) \over 1+yw}  \right) = {xE(x,y)D(x,y,w)
\over 1+xE(x,y)}.
\end{equation}

Let now $B^\bullet_k(x,y)$ be the GF for 2-connected
series-parallel graphs, where the root bears no label and has
degree $k$, and where $y$ marks edges. Then we have the following
relation.

\begin{lem}\label{leB-SP}
$$
w {\partial B^\bullet(x,y,w) \over \partial w} = xyw e^{S(x,y,w)}.
$$
\end{lem}

\begin{proof}
We have $ w \partial B^\bullet(x,y,w) / \partial w = \sum_{k \ge
1} k B^\bullet_k(x,y) w^k.$ The last summation enumerates rooted
2-connected graphs with a distinguished edge incident to the root,
and of these there as many as networks containing the edge between
the poles (this corresponds to the term $e^{S(x,y,w)})$. The
degree of the root in a 2-connected graph corresponds to the
degree of the first pole in the corresponding network, hence the
equation follows.
\end{proof}

From the previous equation it  follows that
\begin{equation}\label{int-sp}
B^\bullet(x,y,w) = xy \int e^{S(x,y,w)} dw.
\end{equation}
Our next task is to get rid of the integral and to express
$B^\bullet$ in terms of $D$. Recall that $E(x,y) = D(x,y,1)$.

\begin{lem}\label{le:sp}
The generating function of rooted 2-connected series-parallel
graphs is equal to
$$
    B^\bullet(x,y,w) = x\left(D(x,y,w) - {xE(x,y) \over 1+xE(x,y)} D(x,y,w) \left(1 + {D(x,y,w) \over 2} \right)
    \right).
    $$
\end{lem}

\begin{proof}
We use the techniques developed in \cite{gn,SP} in order to
integrate (\ref{int-sp}) in closed-form.
$$
\int e^S dw = \int {1+D \over 1+yw} \,dw= y^{-1} \log(1+yw) + \int
{D \over 1+yw}\ dw
$$
Now we integrate by parts and
$$
\int {D \over 1+yw}\ dw = y^{-1} \log(1+yw)D - \int y^{-1}
\log(1+yw) {\partial D \over \partial w} \,dw.
$$
For the last integral we change variables $t=D(x,y,w)$ and use the
fact that $\log(1+yw) = \log(1+t) - xEt/(1+xE)$. We obtain
$$
\int \log(1+yw) {\partial D \over \partial w} \,dw
 = \int_0^D \log(1+t)\,dt - {xE \over 1+xE} \int_0^D t \,dt.
 $$
 Now everything can be integrated in closed form and, after a
 simple manipulation, we obtain the result as claimed.
\end{proof}

In order to prove the main results in this section we need the
singular expansions of $D(x,y)$ and $B(x,y)$ , for a fixed value
of $y$, near the dominant singularity $R(y)$.

\begin{lem}\label{Le4.2}
For $|w| \le 1$ and for fixed $y$ (sufficiently close to $1$) the
dominant singularity of the functions $E(x,y)$, $D(x,y,w)$, and
$B^\bullet(x,y,w)$ (considered as functions in $x$) is given by $x
= R(y)$, where $R(y)$ is an analytic function in $y$ with $R =
R(1) \approx 0.1280038$. Furthermore, we have the following local
expansion:
\begin{align*}
E(x,y) &= E_0(y) + E_1(y) X + E_2(y) X^2 + \cdots,\\
D(x,y,w)&= D_0(y,w) + D_1(y,w) X + D_2(y,w) X^2 + \cdots, \\
B^\bullet(x,y,w)&= B_0(y,w) + B_1(y,w) X + B_2(y,w) X^2 + \cdots,
\end{align*}
where $X = \sqrt{1 - x/R(y)}$.

The functions $R(y)$, $E_j(y)$, $D_j(y,w)$,
and $B_j(y,w)$ are analytic in $y$ resp.\ in $w$ and satisfy the
relations
\begin{align*}
\frac{E_0(y)^3}{E_0(y)-1} &=
\left( \log \frac{1+E_0(y)}{1+R(y)} - E_0(y) \right)^2,\\
R(y) &= \frac{\sqrt{1-1/E_0(y)} -1 }{E_0(y)},\\
E_1(y) &= - \left( \frac {2R(y)E_0(y)^2(1+R(y)E_0(y))^2}
{(2R(y)E_0(y)+R(y)^2E_0(y)^2)^2 + 2R(y)(1+R(y)E_0(y))} \right)^{ 1/2},\\
\end{align*}
\begin{align*}
D_0(y,w) &= (1+yw) \exp\left( {\frac {R(y)E_0(y)}{1+R(y)E_0(y)}D_0(y,w)}\right) - 1,\\
D_1(y,w) &= - \frac{D_0(y,w) E_1(y) R(y) (D_0(y,w)+1)}{(R(y)E_0(y)D_0(y,w)-1)(1+R(y)E_0(y))}, \\
B_0(y,w) &=  -\frac{R(y)D_0(y,w)(R(y)E_0(y)D_0(y,w)-2)}{2(1+R(y)E_0(y))} ,\\
B_1(y,w) &= \frac{E_1(y)R(y)^2D_0(y,w)^2}{2(1+R(y)E_0(y))^2}.
\end{align*}
\end{lem}

\begin{proof}
Since $E(x,y)$ satisfies Equation (\ref{eq:E-sp}) it follows that
the dominant singularity of 
$E(x,y)$ is of square-root type and there is an expansion of the
form $E(x,y) = E_0(y) + E_1(y) X + O(X^2)$, with $X =
\sqrt{1-x/R(y)}$, and where $R(y)$ and $E_j(y)$ are analytic in
$y$; compare with \cite{Ben74,Drm97}. Furthermore, if we set
\[
\Phi(x,y,z) =
(1+y)\exp\left( {x z^2 \over 1+xz}\right) -z-1
\]
then $R(y)$ and $E_0(y)$ satisfy the two equations
\[
\Phi(R(y),y,E_0(y)) = 0 \quad\mbox{and}\quad
\Phi_z(R(y),y,E_0(y)) = 0
\]
and $E_1(y)$ is then given by
\[
E_1(y) = - \left( \frac{2R(y)\Phi_x(R(y),y,E_0(y))}
{\Phi_{zz}(R(y),y,E_0(y))} \right)^{1/2}.
\]

Next observe that for $|w|\le 1$ the radius of convergence of the
function $x\mapsto D(x,y,w)$ is surely $\ge |R(y)|$. However,
$D(x,y,w)$ satisfies Equation (\ref{eq:D-sp}), which implies that
the dominant singularity of $E(x,y)$ carries over to that of
$D(x,y,w)$. Thus, the mapping $x\mapsto D(x,y,w)$ has dominant
singularity $R(y)$ and it also follows that $D(x,y,w)$ has a
singular expansion of the form $D(x,y,w)= D_0(y,w) + D_1(y,w) X +
 O(X^2)$. Hence, by Lemma~\ref{leB-SP} we also get an expansion
for $B^\bullet(x,y,w)$  of that form. Finally the relations for
$D_0,D_1$ and $B_0,B_1$ follow by comparing coefficients in the
corresponding expansions.
\end{proof}

\begin{thm}\label{th:twoconnectedsp}
Let $d_k$ be the limit probability that a vertex of a two-connected
series-parallel graph has degree $k$. then
 $$
 p(w) = \sum_{k\ge 1} d_k w^k = \frac{B_1(1,w)}{B_1(1,1)}.
 $$

Obviously, $p(1)=1$, so that the $d_k$ are indeed a probability
distribution. We have asymptotically, as $k\to\infty$,
\[
d_k \sim c\cdot k^{- 3/2} q^k,
\]
where $c \approx 3.7340799$ is a computable constant and
\[
q  = \left( \left( 1 + 1/({R(1)E_0(1))}\right) e^{- 1/(1 +
R(1)E_0(1))} - 1\right)^{-1} \approx 0.7620402.
\]
\end{thm}

\begin{proof}
First observe that
\[
p(w) = \lim_{n\to\infty} \frac{[x^n]B^\bullet(x,1,w)}{[x^n]B^\bullet(x,1,1)}.
\]
However, from the local expansion of $B^\bullet(x,1,w)$ that is
given in Lemma~\ref{Le4.2} (and by the fact that
$B^\bullet(x,1,w)$ can be analytically continued to a
$\Delta$-region; see Section \ref{sec:prelim}) it follows that
\[
[x^n]B^\bullet(x,1,w) = - \frac{B_1(1,w)}{2\sqrt\pi} n^{- 3/2}
R(1)^{-n} \left( 1 + O\left( \frac 1n\right) \right).
\]
Hence, $p(w) = {B_1(1,w)}/{B_1(1,1)}$.

Next observe that Lemma~\ref{Le4.2} provides $B_1(1,w)$ only for
$|w|\le 1$. However, it is easy to continue $B_1(1,w)$
analytically to a larger region and it is also possible to
determine the dominant singularity of $B_1(1,w)$, from which we
deduce an asymptotic relation for the coefficients of $p(w) =
{B_1(1,w)}/{B_1(1,1)}$.

For this purpose  first observe from Lemma~\ref{Le4.2} that
$D_0(y,w)$ satisfies a functional equation which provides an
analytic continuation of the mapping $w\mapsto D_0(y,w)$ to a
region including the unit disc. In addition, it follows that there
exists a dominant singularity $w_0(y)$ and a local expansion of
the form
\[
D_0(y,w) = D_{00}(y) + D_{01}(y) W + D_{02}(y) W^2 + \cdots,
\]
where $W = \sqrt{1-w/w_0(y)}$. Furthermore, if we set
\[
\Psi(y,w,z) =  (1+yw) \exp\left({\frac
{R(y)E_0(y)}{1+R(y)E_0(y)}z}\right) - z-1
\]
then $w_0(y)$ and $D_{00}(y)$ satisfy the equations
\[
\Psi(y,w_0(y),D_{00}(y)) = 0 \quad\mbox{and}\quad
\Psi_z(y,w_0(y),D_{00}(y)) = 0.
\]
Hence
\[
D_{00}(y) = \frac 1{R(y)E_0(y)} \quad\mbox{and}\quad w_0(y) =
\frac 1y\left( 1 + \frac 1{R(y)E_0(y)}\right) \exp\left(- \frac
1{1 + R(y)E_0(y)}\right) - \frac 1y.
\]

Finally, with the help of  Lemma~\ref{Le4.2} it also follows that
this local representation of $D_0(y,w)$ provides similar local
representations for $D_1$, $B_0$, and $B_1$:
\begin{align*}
D_1(y,w) &= D_{1,-1}(y)W^{-1} + D_{10}(y) + D_{11}(y) W + \cdots,
\\ B_0(y,w) &= B_{00}(y) + B_{02}(y) W^2 + B_{03}(y) W^3 + \cdots,
\\
B_1(y,w) &= B_{10}(y) + B_{11}(y) W + B_{12}(y) W^2 + \cdots,
\end{align*}
where $W = \sqrt{1-w/w_0(y)}$ is as above. Hence,  all functions
of interest $D_0,D_1,B_0,B_1$ can be analytically continued to a
$\Delta$-region, and the asymptotic relation for $d_k$ follows
immediately. Since $w_0(1)$ is the dominant singularity, we have
$q = 1/w_0(1)$.
\end{proof}

The next theorem provides the degree distribution in series-parallel
graphs. This result has been obtained independently in
\cite{angelika2}, and again our respective results agree.

\begin{thm}\label{th:sp}
Let $d_k$ be the limit probability that a vertex of a connected
series-parallel graph has degree $k$. then
 $$
 p(w) = \sum_{k \ge 1} d_k w^k = \rho \cdot {\partial \over
 \partial x} \, e^{B^\bullet(x,1,w)} \left|_{x=\rho C'(\rho)}\right.,
 $$
where $B^\bullet$ is given by Lemma \ref{le:sp} and Equations
(\ref{eq:D-sp}) and (\ref{eq:E-sp}).

Moreover $p(1)=1$, so that the $d_k$ are indeed a probability
distribution. We have asymptotically, as $k\to\infty$,
\[
d_k \sim c\cdot k^{- 3/2} q^k,
\]
where $c \approx 3.5952391$ is a computable constant and
\[
q  = \left( \left( 1 + 1/({\tau E(\tau,1)})\right) e^{- 1/\tau
E(\tau,1)} - 1\right)^{-1} \approx 0.7504161.
\]
\end{thm}

\begin{proof}
The proof of the first statement is exactly the same as for
Theorem \ref{th:outer}. Again, we know  that $\rho C'(\rho)=\tau
\approx 0.127$ is larger than the radius of convergence $\rho
\approx 0.110$ of $C(x)$, so that Lemma~\ref{sing} applies. The
proof that $p(1)=1$ is also the same.

Recall that $\tau < R(1)$. Hence the dominant singularity $x = R(1)$ of
the mapping $x\mapsto B^\bullet(x,1,w)$ will have no influence
to the analysis of $p(w)$. Nevertheless, since
 $$ {\partial \over
 \partial x} \, e^{B^\bullet(x,1,w)} = e^{B^\bullet(x,1,w)}
 {\partial B^\bullet(x,1,w) \over  \partial x}
 $$
we have to get some information on $D(x,1,w)$ and its derivative
${\partial D(x,1,w) /  \partial x}$ with $x = \tau$.

Let us start with the analysis of the mapping $w\mapsto
D(\tau,1,w)$. Since $D(x,y,w)$ satisfies Equation (\ref{eq:D-sp})
it follows that $D(\tau,1,w)$ satisfies
\[
D(\tau,1,w) = (1+w) \exp\left( {\tau E(\tau,1) D(\tau,1,w)
\over 1+\tau E(\tau,1)} \right) -1.
\]
Hence  there exists a dominant singularity $w_1$ and a singular
expansion of the form
\[
D(\tau,1,w) = \widetilde D_0 + \widetilde D_1 \widetilde W +
\widetilde D_2 \widetilde W^2 + \cdots,
\]
where $\widetilde W = \sqrt{1-w/w_1}$. Furthermore, if we set
\[
\Xi(w,z) = (1+w) \exp\left( {\tau E(\tau,1) z
\over 1+\tau E(\tau,1)} \right) -z-1
\]
then $w_1$ and $\widetilde D_0$ satisfy the equations
\[
\Xi(w_1,\widetilde D_0) = 0 \quad\mbox{and}\quad
\Xi_z(w_1,\widetilde D_0) = 0.
\]
Consequently,
\[
\widetilde D_0 = \frac 1{\tau E(\tau,1)} \quad\mbox{and}\quad w_1
= \left( 1 + \frac 1{\tau E(\tau,1)}\right) \exp\left(- \frac 1{1
+ \tau E(\tau,1)}\right) - 1.
\]

Next, by taking derivatives with respect to $x$ in
(\ref{eq:D-sp}) we obtain the relation
\[
{\partial D(x,1,w) \over  \partial x} =
\frac{(1+D(x,1,w))D(x,1,w)(E(x,1) + xE_x(x,1)} {(xE(x,1)D(x,1,w)
-1)(1+xE(x,1))}.
\]
Thus if we set $x= \tau$ and insert the singular representation of
$D(\tau,1,w)$, it follows that ${\partial D(x,1,w) \over  \partial
x} \left|_{x=\tau}\right.$ has a corresponding singular
representation too. By Lemma~\ref{le:sp} we get the same property
for ${\partial B^\bullet(x,1,w) \over  \partial
x}\left|_{x=\tau}\right.$ and finally for
\[
\rho \cdot {\partial \over
 \partial x} \, e^{B^\bullet(x,1,w)} \left|_{x=\tau}\right.
= \widetilde C_0 + \widetilde C_1 \widetilde W +  \widetilde C_2
\widetilde W^2 + \cdots.
\]
This implies the asymptotic relation for $d_k$ with $q = 1/w_1$.
\end{proof}

In this case, we obtain an expression for $p(w)$ in terms of the
functions $E(x,1)$ and $D(x,1,w)$ and their derivatives. The
derivatives can be computed using Equations (\ref{eq:E-sp}) and
(\ref{eq:D-sp}) as in the previous proof. Expanding $p(w)$ in
powers of $w$ we obtain the approximate values for small $k$ shown
in Table \ref{taula-deg}.

\paragraph{Remark.} We have $d_1 = \rho$. Also, there is an easy relation between $d_1$ and $d_2$,
namely
$$
    d_2 = d_1 (2\kappa),
$$
where $\kappa n$ is asymptotically the expected number of edges in
series-parallel graphs. This is shown in \cite[Thm. 4.10]{MSW} for
planar graphs, phrased in terms of the average degree; the only
property required is that subdividing an edge preserves planarity,
which is also true in the case of series-parallel graphs (but not
for outerplanar graphs). The value of $\kappa \approx 1.61673$ was
determined in \cite{SP} and one can check that the relation holds.

\section{Quadrangulations and 3-connected planar graphs}\label{sec:maps}

From now on and for the rest of the paper, all generating
functions are associated to planar graphs. The goal of this
section is to find the generating function of 3-connected planar
graphs according to the degree of the root. This is an essential
ingredient in the next section.

First we work out the problem for simple quadrangulations, which
are in bijection with 3-connected maps. In order to do that we
must revisit the classical work of Brown and Tutte \cite{BT} on
2-connected (non-separable) maps. Finally, using the fact that a
3-connected planar graph has a unique embedding in the sphere, we
finish the job.

\subsection{Simple quadrangulations}

A \emph{rooted quadrangulation} is a planar map where every face
is a quadrangle, and with a distinguished directed edge of the
external face, which is called the \emph{root edge} of the
quadrangulation. The \emph{root vertex} of the quadrangulation is
the tail of the root edge. A \emph{diagonal} is an internal path
of length 2 joining two opposite vertices of the external face. A
quadrangulation is \emph{simple} if it has no diagonal, every
cycle of length 4 other than the external one defines a face, and
it is not the trivial map reduced to a single quadrangle. In
Section~5 of \cite{MS} it is shown how to count simple
quadrangulations. Here we extend this result to  count them also
according to the degree of the root vertex.

A quadrangulation is bipartite and connected, so if we fix the
colour of the root vertex there is a unique way of 2-colouring the
vertices. We call the two colours black and white, and we assume
that the root is black. Diagonals are called black or white
according to the colour of the external vertices they join.

Let $F(x,y,w)$ be the GF of rooted quadrangulations, where the
variables $x$, $y$ and $w$ mark, respectively, the number of black
vertices minus one, the number of white vertices minus one, and
the degree of the root vertex minus one. Generating functions for
maps are always ordinary, since maps are unlabelled objects.

The generating functions $F_{N}$, $F_B$ and $F_W$  are associated,
respectively, to quadrangulations with no diagonal, to those with
at least one black diagonal (at the root vertex), and to those
with at least one white diagonal (not at the root vertex). By
planarity only one of the two kinds of diagonals can appear in a
quadrangulation; it follows that
$$
  F(x,y,w) = F_{N}(x,y,w) + F_B(x,y,w) + F_W(x,y,w).
$$
A quadrangulation with a diagonal can be  decomposed into two
quadrangulations, by considering the maps to the left and to the
right of this diagonal.  This gives rise to the equations
\begin{align*}
  F_B(x,y,w) &= \left(F_N(x,y,w)+F_W(x,y,w)\right)\frac{F(x,y,w)}{x}, \\
  F_W(x,y,w) &= \left(F_N(x,y,w)+F_B(x,y,w)\right)\frac{F(x,y,1)}{y}.
\end{align*}
In the second case,   only one of the two quadrangulations
contribute to the degree of the root vertex; this is the reason
why the term $F(x,y,1)$ appears. The $x$ and the $y$ in the
denominators appear because the three vertices of the diagonal are
common to the two quadrangulations. Since we are considering
vertices minus one,  we only need to correct the colour that
appears twice at the diagonal. Incidentally, no term $w$ appears
in the equations for the same reason.

Let us write $F=F(x,y,w)$ and $F(1)=F(x,y,1)$. From the previous
equations we deduce that
\begin{eqnarray*}
   F &=& F_N+F_B+F_W = (F_N+F_B)\left(1+\frac{F}{x}\right), \\
   F &=& F_N+F_B+F_W = (F_N+F_W)\left(1+\frac{F(1)}{y}\right),
\end{eqnarray*}
so that
$$
  F+F_N = (F_N+F_B)+(F_N+F_W) = F\left(\frac{1}{1+F/x}+
  \frac{1}{1+F(1)/y}\right),
$$
and finally
\begin{equation}
\label{eq:Fn_F}
  F_N = F\left(\frac{1}{1+F/x}+\frac{1}{1+F(1)/y}-1\right).
\end{equation}

Now we proceed to count simple quadrangulations. We use the
following combinatorial decomposition of quadrangulations with no
diagonals in terms of simple quadrangulations: all
quadrangulations with no diagonals, with the only exception of the
trivial one, can be decomposed uniquely into a simple
quadrangulation $q$ and as many quadrangulations as internal faces
$q$ has (replace every internal face of $q$ by its corresponding
quadrangulation).

Let
$$
Q(x,y,w) = \sum_{i, j, k} q_{i, j, k} x^i y^j w^k
$$
be the GF of simple quadrangulations, where $x$, $y$ and $w$ have
the same meaning as for $F$. We notice that this GF is called
$Q_N^{*}$ in~\cite{MS}. We translate the  combinatorial
decomposition of simple quadrangulations into generating functions
as follows.
\begin{align}
\notag F_N(x,y,w) - xyw &= \sum_{i, j, k} q_{i, j, k}
                                x^i y^j
                                \left(\frac{F}{xy}\right)^k
                                \left(\frac{F(1)}{xy}\right)^{i+j-1-k} \\
\notag
                 &= \sum_{i, j, k} q_{i, j, k}
                                \frac{xy}{F(1)}
                                \left(\frac{F(1)}{y}\right)^i
                                \left(\frac{F(1)}{x}\right)^j
                                \left(\frac{F}{F(1)}\right)^k = \\
\label{eq:Fn_Q}
                 &= \frac{xy}{F(1)}
                    Q\left(\frac{F(1)}{y}, \frac{F(1)}{x},
                    \frac{F}{F(1)}\right),
\end{align}
where we are using the fact that a quadrangulation counted by
$q_{i, j, k}$ has $i+j+2$ vertices, $i+j-1$ internal faces, and
$k$ of them are incident to the root vertex.

At this point we change variables as $X=F(1)/y$, $Y=F(1)/x$ and
$W=F/F(1)$. Then Equations~(\ref{eq:Fn_F}) and (\ref{eq:Fn_Q}) can
be rewritten  as
$$
\frac{xy}{F(1)} \,Q(X,Y,W) = F_N - xyw =
F\left(\frac{1}{1+F/x}+\frac{1}{1+F(1)/y}-1\right) - xyw,
$$
\begin{equation}
                \label{eq:Q_XYW}
                Q(X,Y,W) = XYW\left(\frac{1}{1+WY}+\frac{1}{1+X}-1\right) - F(1)w.
\end{equation}
The last equation would be an explicit expression of $Q$ in terms
of $X, Y, W$ if it were not for the term $F(1)w=F(x,y,1)w$. In
\cite{MS} it is shown that
\begin{equation}\label{eq:F1_RS}
F(1) = \frac{RS}{(1+R+S)^3},
\end{equation}
where $R=R(X,Y)$ and $S(X,Y)$ are algebraic functions defined by
\begin{equation}\label{eq:RS}
 R=X(S+1)^2, \qquad S=Y(R+1)^2.
\end{equation}
Hence it remains only to obtain an expression for $w=w(X,Y,W)$ in
order  to obtain an explicit expression for $Q$. This is done in
the next subsection.

\subsection{Rooted non-separable planar maps}

In  \cite{BT} the authors  studied the generating function
$h(x,y,w)$ of rooted non-separable planar maps where $x$, $y$ and
$w$ count, respectively the number of vertices minus one, the
number of faces minus one, and the valency (number of edges) of
the external face. We notice that the variable $z$ is used instead
of $w$ in~\cite{BT}.  There is a bijection between rooted
quadrangulations and non-separable rooted planar maps: black and
white vertices in the quadrangulation correspond, respectively, to
faces and vertices of the map; quadrangles become edges; and the
root vertex becomes the external face, and its degree becomes its
valency. As a consequence  $h(y,x,w) = wF(x, y, w) $, where the
extra factor $w$ appears because in $F$ we are counting the degree
of the root vertex minus one. It follows that Equation~(3.9) from
\cite{BT} becomes
$$(1-w)(1-yw)wF=-w^2F^2+(-xw+wF(1))wF+xw^2(x(1-w)+F(1)).$$
By dividing both sides by $F(1)^2$ and rewriting in terms of
$X=F(1)/y$, $Y=F(1)/x$ and $W=F/F(1)$, we obtain
$$
(1-w)\left(\frac{1}{F(1)}-\frac{w}{X}\right)wW =
-w^2W^2+\left(1-\frac{1}{Y}\right)w^2W
+\frac{w^2}{Y}\left(\frac{1}{X}(1-w)+1\right),
$$
$$
\frac{(1-w)(X-wF(1))wW}{XF(1)} =
\frac{w^2(-XYW^2+XYW-XW+1-w+X}{XY},
$$
\begin{equation}\label{eq:w_XYW}
Y(1-w)(X-wF(1))W = wF(1)(-XYW^2+XYW-XW+1-w+X).
\end{equation}
Observe that this is a quadratic equation in $w$. Solving for $w$
in (\ref{eq:w_XYW}) and using (\ref{eq:F1_RS}) and (\ref{eq:RS})
we get (the plus sign is because $T^\bullet$ has positive
coefficients in coming Theorem \ref{3-conn})
\begin{equation}
\label{eq:w_RSW} w = \frac{ -w_1(R,S,W)+(R-W+1)\sqrt{w_2(R,S,W)}
}{ 2(S+1)^2(SW+R^2+2R+1) },
\end{equation}
where $w_1(R,S,W)$ and $w_2(R,S,W)$ are polynomials given by

\begin{align}\label{eqW1}
w_1 = & -RSW^2+W(1+4S+3RS^2+5S^2+R^2+2R+2S^3+3R^2S+7RS) \\
      & +(R+1)^2(R+2S+1+S^2), \notag \\
      \label{eqW2}
w_2 = & R^2S^2W^2-2WRS(2R^2S+6RS+2S^3+3RS^2+5S^2+R^2+2R+4S+1) \\
      & +(R+1)^2(R+2S+1+S^2)^2. \notag
\end{align}
The reason we choose to write $w$ as a function of  $(R,S,W)$
instead of $(X,Y,W)$ will become clear later on.

Thus, together with Equations~(\ref{eq:Q_XYW}) and
(\ref{eq:F1_RS}), we have finally obtained an explicit expression
for the generating function  $Q(X,Y,W)$ of simple quadrangulations
in terms of $W$ and algebraic functions $R(X,Y)$ and $S(X,Y)$.

\subsection{3-connected planar graphs}

Let $T^\bullet(x,z,w)$ be the GF of 3-connected planar graphs,
where one edge is taken as the root  and given a direction, and
where $x$ counts vertices, $z$ counts edges, and $w$ counts the
degree of the tail of the root edge. Now we relate  $T^\bullet$ to
the  GF $Q(X,Y,W)$ of simple quadrangulations.

By the bijection between simple quadrangulations and 3-connected
planar maps, and using Euler's relation, the GF $xwQ(xz,z,w)$
counts rooted 3-connected planar maps, where $z$ marks edges (we
have added an extra term $w$ to correct the `minus one' in the
definition of $Q$).

According to Whitney's theorem 3-connected planar graphs have a
unique embedding in the sphere. As noticed in \cite{bender}, the
are two ways of rooting an embedding of a directed edge-rooted
graph in order to get a rooted map, since there are two ways of
choosing the root face adjacent to the root edge. It follows that
\begin{equation}
\label{eq:T_z_short}
 T^\bullet(x,z,w) = \frac{xw}{2}\,Q(xz,z,w).
\end{equation}

\begin{thm}\label{3-conn}
The generating function of directed edge-rooted 3-connected planar
graphs, where $x,z,w$ mark, respectively, vertices, edges, and the
degree of the root vertex, is equal to
\begin{equation}
\label{eq:T_z_long} T^\bullet =
\frac{x^2z^2w^2}{2}\left(\frac{1}{1+wz}+\frac{1}{1+xz}-1 - \frac{
(u+1)^2 \left(-w_1(u,v,w)+(u-w+1)\sqrt{w_2(u,v,w)}\right) }{
2w(vw+u^2+2u+1)(1+u+v)^3 }\right),
\end{equation}
where  $u$ and $v$ are algebraic functions defined by
\begin{equation}\label{eq:uv}
u=xz(1+v)^2, \qquad v=z(1+u)^2,
\end{equation}
and $w_1(u,v,w)$ and $w_2(u,v,w)$ are given by (\ref{eqW1}) and
(\ref{eqW2}) replacing $R,S,W$ by $u,v,w$, respectively.
\end{thm}

\begin{proof}
Combine Equation (\ref{eq:T_z_short}), together with Equations
(\ref{eq:Q_XYW}), (\ref{eq:F1_RS}),  and (\ref{eq:w_RSW}).
\end{proof}

When we set $w=1$ in Equation~(\ref{eq:T_z_long}) we recover the
GF of edge-rooted 3-connected planar graphs without taking into
account the degree of the root vertex.

\section{Planar graphs}\label{sec:planar}

This section is divided into three parts. First we obtain an
explicit expression for $B^\bullet(x,y,w)$, the generating
function of rooted 2-connected planar graphs taking into account
the degree of the root. Secondly, we compute singular expansions
at dominant singularities for several generating functions. And
finally we obtain the asymptotic degree distribution in random
planar graphs.

\subsection{2-connected planar graphs}

Let  $B^\bullet(x,y,w)$, the generating function of rooted
2-connected planar graphs taking into account the degree of the
root. As for series series-parallel graphs we have to work with
networks.

Let $T^\bullet(x,z,w)$ be the GF for directed edge-rooted
3-connected planar maps as in the previous section. As in Section
\ref{sec:sp}, we denote by $D(x,y,w)$ and $S(x,y,w)$,
respectively, the GFs of (planar) networks and series networks,
with the same meaning for the variables $x,y$ and $w$.

\begin{lem}\label{D-planar}
We have  
\begin{eqnarray*}
D(x,y,w) &=& (1+yw)\exp\left(S(x,y,w)+ {1 \over x^2D(x,y,w)}
 T^\bullet\left(x,E(x,y),{ D(x,y,w) \over E(x,y)}\right)\right)  -1 \\ S(x,y,w) &=& x
E(x,y)\left( D(x,y,w) - S(x,y,w)\right),
\end{eqnarray*}
where $E(x,y) = D(x,y,1)$ is the GF for planar networks (without
marking the degree of the root).
\end{lem}

\begin{proof}
The proof is a variant of the equations developed by Walsh
\cite{walsh}, taking into account the degree of the first pole in
a network. The main point is the substitution of variables in
$T^\bullet$: an edge is substituted by an ordinary  planar network
(this accounts for the term $E(x,y)$), except if it is incident
with the first pole, in which case it is substituted by a planar
network marking the degree, hence the term $D(x,y,w)$ (it is
divided by $E(x,y)$ in order avoid overcounting of ordinary
edges).
\end{proof}

As in Lemma \ref{leB-SP}, and for the same reason, we have
\begin{align*}
w {\partial B^\bullet(x,y,w) \over \partial w} &= \sum_{k \ge 1} k
B_k^\bullet(x,y) w^k \\
& = xyw \exp\left(S(x,y,w)+ {1 \over x^2D(x,y,w)}
 T^\bullet\left(x,E(x,y),{ D(x,y,w) \over E(x,y)}\right)\right).
\end{align*}

\begin{lem}\label{B-planar}
The generating function of rooted 2-connected planar graphs is
equal to
\begin{equation}
B^\bullet(x,y,w) = x\left(D - {xED \over 1+xE}  \left(1 + {D \over
2} \right) \right)  - {1+ D \over xD }
    T^\bullet(x,E,D/E) + {1 \over x} \int_0^{D} {T^\bullet(x,E,t/E) \over t}\, dt,
\end{equation}
where for simplicity we let $D=D(x,y,w)$ and $E=E(x,y)$.
\end{lem}

\begin{proof}
We start as in the proof of Lemma \ref{le:sp}.
$$
\int {1+D \over 1+yw} \,dw= y^{-1} \log(1+yw)
+  y^{-1} \log(1+yw)D - \int y^{-1}\log(1+yw) {\partial D \over
\partial w} \,dw.
$$
For the last integral we change variables $t=D(x,y,w)$ and use the
fact that
$$
\log(1+yw) = \log(1+D) - {xED \over (1+xE)} - {1 \over x^2D}
T^\bullet(x,E,D/E).
 $$
 We obtain
$$
\int \log(1+yw) {\partial D \over \partial w} \,dw
 = \int_0^D \log(1+t)\,dt - {xE \over 1+xE} \int_0^D t \,dt+ {1\over x^2}\int_0^D {T^\bullet(x,E,t/E) \over t}dt.
 $$
On the right-hand side, all the integrals except the last one are
elementary. Now we use
$$
B^\bullet(x,y,w) = xy \int {1+D \over 1+yw} \,dw
$$
and after a simple manipulations the result follows.
\end{proof}

In order to get a full expression for $B^\bullet(x,y,w)$, it
remains to compute the integral in the formula of the previous
lemma.

\begin{lem}\label{integral}
Let $T^\bullet(x,z,w)$ be the GF of 3-connected planar graphs as
before. Then
\begin{eqnarray*}
 &&\int_0^w {T^\bullet(x,z,t) \over t}\, dt =
-\frac{x^2(z^3xw^2-2wz-2xz^2w+(2+2xz)\log(1+wz))} {4(1+xz)} \\
&& - \frac{uvx} {2(1+u+v)^3} \left(
\frac{w(2u^3+(6v+6)u^2+(6v^2-vw+14v+6)u+4v^3+10v^2+8v+2)}
{4v(v+1)^2}
\right. \\
&& +\frac{(1+u)(1+u+2v+v^2) (2u^3+(4v+5)u^2+(3v^2+8v+4)u
+2v^3+5v^2+4v+1)}{4uv^2(v+1)^2} \\
&&
-\frac{\sqrt{Q}(2u^3+(4v+5)u^2+(3v^2-vw+8v+4)u+5v^2+2v^3+4v+1)}{4uv^2(v+1)^2} \\
&& +\frac{(1+u)^2(1+u+v)^3\log(Q_1)}{2v^2(1+v)^2} \\
&& \left.
+\frac{(u^3+2u^2+u-2v^3-4v^2-2v)(1+u+v)^3\log(Q_2)}{2v^2(1+v)^2u}
\right),
\end{eqnarray*}

where the expressions $Q$, $Q_1$ and $Q_2$ are given by

\begin{align*}
  Q &= u^2v^2w^2 - 2uvw( u^2(2v+1)+u(3v^2+6v+2)+2v^3+5v^2+4v+1) \\
    & +(1+u)^2(u+(v+1)^2)^2 \\
  Q_1 &= \frac{1}{2(wv+(u+1)^2)^2(v+1)(u^2+u(v+2)+(v+1)^2)}\left( -uvw(u^2+u(v+2)+2v^2+3v+1)\right. \\
      &  \left. +(u+1)(u+v+1)\sqrt{Q}+(u+1)^2(2u^2(v+1)
                +u(v^2+3v+2)+v^3+3v^2+3v+1 \right) \\
  Q_2 &= \frac{-wuv+u^2(2v-1)+u(3v^2+6v+2)+2v^3+5v^2+4v+1-\sqrt{Q}}{2v(u^2+u(v+2)+(v+1)^2)}
\end{align*}

\end{lem}

\begin{proof}
We use Equation~(\ref{eq:T_z_long}) to integrate
$T^\bullet(x,z,w)/w$. Notice that neither $u$ nor $v$ have any
dependence on $w$. We have used Maple to obtain a primitive, to
which we have added the appropriate constant $c(x,z)$ to ensure
that the resulting expression evaluates to $0$ when $w=0$.

A key point in the previous derivation is that, by expressing
$w(X,Y,W)$ (see Equation~(\ref{eq:w_XYW})) in terms of $R, S$
instead of $X,Y$ (see Equation~(\ref{eq:w_RSW})), we obtain a
quadratic polynomial $w_2(R,S,W)$ in terms of $W$ inside the
square root of $T^\bullet(x,z,w)$ in Equation~(\ref{eq:T_z_long}).
Otherwise, we would have obtained a cubic polynomial inside the
square root, and the integration would have been much harder.
\end{proof}

Combining Lemmas \ref{B-planar} and \ref{integral} we can produce
an explicit (although quite long) expression for
$B^\bullet(x,y,w)$ in terms of $D(x,y,w)$, $E(x,y)$, and the
algebraic functions $u(x,y), v(x,y)$. This is needed in the next
section for computing the singular expansion of $B^\bullet(x,y,w)$
at its dominant singularity.

\subsection{Singular expansions}

In this section we find singular expansions of $T^\bullet(x,z,w),
D(x,y,w)$ and $B^\bullet(x,y,w)$ at their dominant singularities.
As we show here, these singularities do not depend on $w$ and were
found in \cite{bender} and \cite{gn}. But the coefficients of the
singular expansions do depend  on $w$, and we need to compute them
exactly in each case.
In the next section we need the singular expansion for
$B^\bullet$, but to compute it we first need the singular
expansions of $u$, $v$, $T^\bullet$ and $D$ (for $u$ and $v$, see
also \cite{benderRich,bender}).

\begin{lem}\label{lem:uv}
Let $u=u(x,z)$ and $v=v(x,z)$ be the solutions of the system of
equations $u = xz(1+v)^2$ and $v = z(1+u)^2$. Let $r(z)$ be given
explicitly by
\begin{equation}\label{eqrz}
r(z) = \frac{\widetilde u_0(z)}{z(1+z(1+\widetilde u_0(z))^2)^2},
\end{equation}
where
\[
\widetilde u_0(z) = -\frac 13 + \sqrt{\frac 49 + \frac 1{3z}}.
\]
Furthermore, let $\tau(x)$ be the inverse function of $r(z)$ and
let $u_0(x) = \widetilde u_0(\tau(x))$ which is also the solution
of the equation
\[
x = \frac{(1+u)(3u-1)^3}{16 u}.
\]
Then, for $x$ sufficiently close to the positive real axis
the function $u(x,z)$ and $v(x,z)$ have a dominant singularity
at $z = \tau(x)$ and have local expansions of the form
\begin{align*}
u(x,z) &= u_0(x) + u_1(x)Z + u_2(x) Z^2 + u_3(x) Z^3 + O(Z^4),\\
v(x,z) &= v_0(x) + v_1(x)Z + v_2(x) Z^2 + v_3(x) Z^3 + O(Z^4),
\end{align*}
where $Z = \sqrt{1-z/\tau(x)}$. The functions $u_j(x)$ and
$v_j(x)$ are also analytic and can be given explicitly  in terms
of $u = u_0(x)$. In particular we have
\begin{align*}
 u_0(x) &= u
            &   v_0(x) &= \frac{1+u}{3u-1}, \\
 u_1(x) &= -\sqrt{2u(u+1)}
            &   v_1(x) &= -\frac{2\sqrt{2u(u+1)}}{3u-1}, \\
 u_2(x) &= \frac{(1+u)(7u+1)}{2(1+3u)}
            &   v_2(x) &= \frac{2u(3+5u)}{(3u-1)(1+3u)}, \\
 u_3(x) &= -\frac{(1+u)(67u^2+50u+11)u}{4(1+3u)^2\sqrt{2u^2+2u}}
            &   v_3(x) &=
            -\frac{\sqrt{2}u(1+u)(79u^2+42u+7)}{4(1+3u)^2(3u-1)\sqrt{u(1+u)}}.
\end{align*}

Similarly, for $z$ sufficiently close to the real axis the
functions $u=u(x,z)$ and $v=v(x,z)$ have a dominant singularity $x
= r(z)$ and there is also a local expansion of the form
\begin{align*}
u(x,z) &= \widetilde u_0(z) + \widetilde u_1(z)X + \widetilde u_2(z) X^2 +  O(X^3),\\
v(x,z) &= \widetilde v_0(z) + \widetilde v_1(z)X + \widetilde
v_2(z) X^2 + O(X^3),
\end{align*}
where now $X = \sqrt{1-x/r(z)}$. The functions $\widetilde u_j(z)$
and $\widetilde v_j(z)$ are analytic and can be given explicitly
in terms of $\widetilde u = \widetilde u_0(z)$. In particular, we
have
\begin{align*}
 \widetilde u_0(z) &= \widetilde u
            &   \widetilde v_0(z) &= \frac{1+\widetilde u}{3\widetilde u-1} \\
 \widetilde u_1(z) &= -\frac{2\widetilde u\sqrt{1+\widetilde u}}{\sqrt{1+3\widetilde u}}
            &   \widetilde v_1(z) &= -\frac{4\widetilde u\sqrt{1+\widetilde u}}{(3\widetilde u-1)\sqrt{1+3\widetilde u}} \\
 \widetilde u_2(z) &= \frac{2(1+\widetilde u)\widetilde u(2\widetilde u+1)}{(1+3\widetilde u)^2}
            &   \widetilde v_2(z) &= \frac{4\widetilde u(5\widetilde u^2 + 4\widetilde u+1)}{(3\widetilde u-1)(1+3\widetilde u)^2} \\
 \widetilde u_3(z) &= -\frac{2\widetilde u(10\widetilde u^3+11\widetilde u^2+5\widetilde u+1)\sqrt{1+\widetilde u}}{(1+3\widetilde u)^{7/2}}
            &   \widetilde v_3(z) &-\frac{4\widetilde u(2\widetilde u+1)(11\widetilde u^2+5\widetilde u+1)\sqrt{1+\widetilde u}}{(3\widetilde u-1)(1+3\widetilde u)^{7/2}}
\end{align*}

\end{lem}

\begin{proof}
Since $u(x,z)$ satisfies the functional equation $u =
xz(1+z(1+u)^2)^2$, it follows that for any fixed real and positive
$z$ the function $x\mapsto u(x,z)$ has a square-root singularity
at $r(z)$ that satisfies the equations
\[
\Phi(r(z),z,u) = 0 \quad \mbox{and}\quad
\Phi_u(r(z),z,u) = 0,
\]
where $\Phi(x,z,u) = u - xz(1+z(1+u)^2)^2$. Now a short
calculation gives the explicit formula (\ref{eqrz}) for $r(z)$. By
continuity we obtain the same kind of representation if $z$ is
complex but sufficiently close to the positive real axis.

We proceed in the same way if $x$ is fixed and $z$ is considered
as the variable. Then $\tau(x)$, the functional inverse of $r(z)$,
is the singularity of the mapping $z\mapsto u(x,z)$. Furthermore,
the coefficients $u_1(x)$ etc.\ can be easily calculated. The
derivations for $v(x,z)$ are completely of the same kind.
\end{proof}

\begin{lem} \label{lem:Ts}
Suppose that  $x$ and $w$ are sufficiently close to
the positive real axis and that $|w| \le 1$. Then the dominant singularity
$z=\tau(x)$ of $T^\bullet(x,z,w)$ does not depend on $w$. The
singular expansion at $\tau(x)$ is
\begin{equation}\label{singT}
T^\bullet(x,z,w) = T_0(x,w) + T_2(x,w)Z^2 + T_3(x,w)Z^3 + O(Z^4),
\end{equation}
where $Z = \sqrt{1-z/\tau(x)}$, and the expressions for the $T_i$
are given in the appendix.
\end{lem}

\begin{proof}
Suppose for a moment that all variables $x,z,w$ are non-negative
real numbers and let us look at the expression (\ref{eq:T_z_long})
 for $T^\bullet$.
The algebraic functions $u(x,z)$ and $v(x,z)$ are always
non-negative and, since the factor $w$ in the denominator cancels
with a corresponding factor in the numerator, the only possible
source of singularities are: a) those coming from $u$ and $v$, or
b) the vanishing of $w_2(u,v,w)$ inside the square root.

We can discard source b) as follows. For fixed $u,v>0$, let
$w_2(w) = w_2(u,v,w)$. We can check that
\begin{align*}
  w_2(1) &= (1+2u+u^2+2v+v^2+uv-uv^2)^2, \\
  w_2'(w) &= -2uv\left((6-w)uv + 1+2u+u^2+4v+5v^2+2v^3+3uv^2+2u^2v\right).
\end{align*}

In particular $w_2(1)>0$ and $w_2'(w) < 0 $ for
$w \in [0,1]$, thus it follows that $w_2(w)>0$ in $w\in[0,1]$.  Hence
the singularities come from source a) and do not depend on $w$.

Following Lemma~\ref{lem:uv} (see also \cite{benderRich} and
\cite{bender}), we have that $z=\tau(x)$ is the radius of
convergence of $u(x,z)$, as a function of $z$. Now by using the
expansions of $u(x,z)$ and $v(x,z)$ from Lemma~\ref{lem:uv} we
obtain (\ref{singT}).

Finally, by continuity all properties are also valid if $x$ and
$w$ are sufficiently close to the real axis, thus completing the
proof.
\end{proof}

Similarly we get an alternate representation expanding in the variable
$x$.

\begin{lem} \label{lem:Ts-2}
Suppose that  $z$ and $w$ are sufficiently close to
the positive real axis and that $|w| \le 1$. Then the dominant singularity
$x=r(z)$ of $T^\bullet(x,z,w)$ does not depend on $w$. The
singular expansion at $r(z)$ is
\begin{equation}\label{singT2}
T^\bullet(x,z,w) = \widetilde T_0(z,w) + \widetilde T_2(z,w)X^2 +
\widetilde T_3(z,w)X^3 + O(X^4),
\end{equation}
where $X = \sqrt{1-x/r(z)}$. Furthermore we have
\begin{align*}
\widetilde T_0(z,w) &= T_0(r(z),w),\\
\widetilde T_2(z,w) &= T_2(r(z),w)H(r(z),z) - T_{0,x}(r(z),x)r(z),\\
\widetilde T_3(z,w) &= T_3(r(z),w)H(r(z),z)^{ 3/2},
\end{align*}
where $H(x,z)$ is a non-zero analytic function with
$Z^2 = H(x,z) X^2$.
\end{lem}

\begin{proof}
We could repeat the proof of Lemma~\ref{lem:Ts}. However, we
present an alternate approach that uses the results
of Lemma~\ref{lem:Ts} and a kind of singularity transfer.

By applying the Weierstrass preparation theorem it follows
that there is a non-zero analytic function with $Z^2 = H(x,z) X^2$.
Furthermore, by using the representation
$x = r(z)(1-X^2)$ and Taylor expansion we have
\begin{align*}
H(x,z) &= H(r(z),z) - H_x(r(z),z) r(z) X^2 + O(X^4),\\
T_j(x,w) &= T_j(r(z),w) - T_{j,x}(r(z),w) r(z) X^2 + O(X^4).
\end{align*}
Hence, Lemma (\ref{lem:Ts}) directly gives proves the result. In
fact, $H$ can be computed explicitly and is equal to
$H(x,\tau(x))=(1+3u)/2u$.
\end{proof}

The next result is Proposition 6.3 from \cite{companion}, and is
needed in order to guarantee that singular expansions of the
desired kind exist.

\begin{thm}\label{ProX2}
Suppose that $F(x,y,u)$ has a local representation of the form
\begin{equation}\label{eqPro31}
F(x,y,u) = g(x,y,u) + h(x,y,u)\left( 1- \frac y{r(x,u)}
\right)^{3/2}
\end{equation}
with functions $g(x,y,u)$, $h(x,y,u)$, $r(x,u)$ that are analytic
around $(x_0,y_0,u_0)$ and satisfy
$$g_y(x_0,y_0,u_0)\ne 1, \quad h(x_0,y_0,u_0)\ne 0,
\quad r(x_0,u_0) \ne 0, \quad r_x(x_0,u_0) \ne g_x(x_0,y_0,u_0).
$$
Furthermore, suppose that $y = y(x,u)$ is a solution of the
functional equation
\[
y = F(x,y,u)
\]
with $y(x_0,u_0) = y_0$. Then $y(x,u)$ has a local representation
of the form
\begin{equation}\label{eqPro32}
y(x,u) = g_1(x,u) + h_1(x,u)\left( 1- \frac x{\rho(u)}
\right)^{3/2},
\end{equation}
where $g_1(x,u)$, $h_1(x,u)$ and $\rho(u)$ are analytic at
$(x_0,u_0)$ and satisfy $h_1(x_0,u_0)\ne 0$ and $\rho(u_0) = x_0$.
\end{thm}

\begin{lem}\label{D-planar-sing}
Suppose that  $y$ and $w$ are sufficiently close to
the positive real axis and that $|w| \le 1$.
Then the dominant singularity
$x=R(y)$ of $D(x,y,w)$ does not depend on $w$. The singular
expansion at $R(y)$ is
\begin{equation}\label{singD}
D(x,y,w) = D_0(y,w) + D_2(y,w)X^2 + D_3(y,w)X^3 + O(X^4),
\end{equation}
where $X = \sqrt{1-x/R(y)}$, and the expressions for the $D_i$ are
given in the appendix.
\end{lem}

\begin{proof}
We use the system of equations from Lemma~\ref{D-planar}.
If we set $w=1$ and substitute $S(x,y,1) = x E(x,y)^2/(1+xE(x,y))$
then we get a single equation for $E(x,y) = D(x,y,1)$:
\begin{equation}\label{eq:singleequation}
E(x,y) = (1+y)\exp\left( \frac{xE(x,y)^2}{1+xE(x,y)}
+\frac 1{x^2E(x,y)} T^\bullet(x,E(x,y),1)\right)-1.
\end{equation}
Now we use the singular expansion from Lemma~\ref{lem:Ts-2} and
Theorem~\ref{ProX2} (which is  Proposition 6.3 of
\cite{companion}) to conclude that $E(x,y)$ has an expansion of
the form
\begin{equation}\label{singE1}
E(x,y) = E_0(y) + E_2(y)X^2 + E_3(y)X^3 + O(X^4).
\end{equation}
Finally we reconsider Lemma~\ref{D-planar} and after
substituting $S(x,y,w) = x E(x,y)D(x,y,w)/(1+xE(x,y))$
we get a corresponding (single) equation for $D(x,y,w)$:
\begin{equation}\label{eq:singleequation2}
D(x,y,w) = (1+yw)\exp\left( \frac{xE(x,y)D(x,y,w)}{1+xE(x,y)}
+\frac 1{x^2D(x,y,w)} T^\bullet\left(
x,E(x,y),\frac{D(x,y,w)}{E(x,y)}\right) \right)-1.
\end{equation}
Now a second application of Lemma~\ref{lem:Ts-2} and
Theorem~\ref{ProX2}  yields the result. Note that the singularity
does not depend on $w$.
\end{proof}

\begin{lem}\label{B-planar-sing}
Suppose that  $y$ and $w$ are sufficiently close to
the positive real axis and that $|w| \le 1$. Then the dominant singularity
$x=R(y)$ of $B^\bullet(x,y,w)$ does not depend on $w$, and is the
same as for $D(x,y,w)$. The singular expansion at $R(y)$ is
\begin{equation}\label{singB}
B^\bullet(x,y,w) = B_0(y,w) + B_2(y,w)X^2 + B_3(y,w)X^3 + O(X^4),
\end{equation}
where $X = \sqrt{1-x/R(y)}$, and the expressions for the $B_i$ are
given in the appendix.
\end{lem}

\begin{proof}
We just have to use the representation of $B^\bullet(x,y,w)$
that is given in Lemma~\ref{B-planar} and Lemma~\ref{integral}
and the singular expansion of $D(x,y,w)$ from Lemma~\ref{D-planar-sing}.
\end{proof}

\subsection{Degree distribution  for planar graphs}

We start with the degree distribution in 3-connected graphs, both
for edge-rooted and vertex-rooted graphs.

\begin{thm}\label{th:threeconnectedplanar}
Let $d_k$ be the limit probability that a vertex of a
three-connected planar graph has degree $k$, and let $e_k$ be the
limit probability that the tail root vertex of an edge-rooted
(where the edge is oriented) three-connected planar graph has
degree $k$. Then
 $$
 \sum e_k w^k = \frac{T_3(r,w)}{T_3(r,1)},
 $$
where $T_3(x,w)$ is given in the appendix and $r = r(1) = (7 \sqrt
7 - 17)/32$ is explicitly given by (\ref{eqrz}). Obviously the
$e_k$ are indeed a probability distribution. We have
asymptotically, as $k\to\infty$,
\[
e_k \sim c \cdot k^{1/2} q^k,
\]
where $c \approx 0.9313492$ is a computable constant and $q =
1/(u_0+1) = \sqrt{7}-2$, and where $u_0 =u(r)= (\sqrt{7}-1)/3$.

Moreover, we have
$$
d_k = \alpha {e_k \over k} \sim c \alpha \cdot k^{-1/2} q^k,
$$
where
 $$\alpha = \frac{(3u_0-1)(3u_0+1)(u_0+1)}{u_0} = {\sqrt{7}+7 \over 2}$$
 is the asymptotic value of the expected average degree
in 3-connected planar graphs.
\end{thm}

We remark that the degree distribution in 3-connected planar maps
counted according to the number of edges was obtained in
\cite{bencan}. The asymptotic estimates have the same shape as our
$d_k$, but the corresponding value of $q$ is equal to $1/2$.

\begin{proof}
The proof uses first the singular expansion (\ref{singT2}). The
representation
 $$
  \sum_{k\ge 1} e_k w^k  = \frac{T_3(r,w)}{T_3(r,1)}
 $$
  follows in completely
the same way as the proof of Theorem~\ref{th:twoconnectedsp},
using now Lemma \ref{lem:Ts}, with the difference that now the
dominant term is the coefficient of $Z^3$.

In order to characterize the dominant singularity of $T_3(1,w)$
and to determine the singular behaviour we observe that the
explicit representation for $T_3$ contains in the denominator a
(dominating) singular term of the form $(-w+u+1)^{3/2}$. Hence, it
follows that $u_0+1$ is the dominant singularity and we also get
the proposed asymptotic relation for $e_k$.

For the proof of the second part of the statement,  let $t_{n,k}$
be the number of vertex-rooted graphs with $n$ vertices and with
degree of the root equal to $k$, and let $s_{n,k}$ be the
analogous quantity for edge-rooted graphs. Let also $t_n = \sum_k
t_{n,k}$ and $s_n = \sum_k s_{n,k}$. Since a vertex-rooted graph
with a root of degree $k$ is counted $k$ times as an
oriented-edge-rooted graph, we have $s_{n,k} = k t_{n,k}$
(a~similar argument is used in \cite{liskovets}). Notice that $e_k
= \lim s_{n,k}/s_n$ and $d_k = \lim t_{n,k}/t_n$.

Using the quasi-powers theorems as in \cite{gn}, one shows that
the expected number of edges $\mu_n$ in 3-connected planar graphs
is asymptotically  $\mu_n \sim \kappa n$, where $\kappa =
-\tau'(1)/\tau(1)$, and $\tau(x)$ is as in Lemma \ref{lem:uv}.
Clearly $s_n = 2 \mu_n t_n/n$.
Finally,  $2 \mu_n /n$ is asymptotic to the expected average
degree $\alpha = 2 \kappa$. Summing up, we obtain
$$
    k d_k = \alpha \, e_k.
    $$
A simple calculation gives the value of $\alpha$ as claimed.
\end{proof}

\begin{thm}\label{th:twoconnectedplanar}
Let $d_k$ be the limit probability that a vertex of a two-connected
planar graph has degree $k$. Then
 $$
 p(w) = \frac{B_3(1,w)}{B_3(1,1)},
 $$
where $B_3(y,w)$ is given in the appendix.

Obviously, $p(1)=1$, so that the $d_k$ are indeed a probability
distribution and we have asymptotically, as $k\to\infty$,
\[
d_k \sim c k^{-1/2} q^k,
\]
where $c \approx 3.0826285$ is a computable constant and
\[
q  =
\left(\frac{1}{1-t_0}\exp\left(\frac{(t_0-1)(t_0+6)}{6t_0^2+20t_0+6}\right)
-1\right)^{-1} \approx 0.6734506,
\]
and $t_0 = t(1) \approx 0.6263717$ is a computable constant given
in the appendix.
\end{thm}

\begin{proof}
The representation of $p(w)$ follows from (\ref{singB}).

Now, in order to characterize the dominant singularity of
$B_3(1,w)$ and to determine the singular behaviour we first
observe that the right hand side of the equation for $D_0$
contains a singular term of the form $(D_0(t-1)+t)^{3/2}$ that
dominates the right hand side. Hence, by applying
Theorem~\ref{ProX2} it follows that $D_0(1,w)$ has dominant
singularity $w_3$, where $D_0(1,w_3) = t/(1-t)$ and we have a
local singular representation of the form
\[
D_0(1,w) = \widetilde D_{00} + \widetilde D_{02} \widetilde W^2 +
\widetilde D_{03} \widetilde W^3 + \cdots,
\]
where $\widetilde W = \sqrt{1-w/w_3}$ and $\widetilde D_{00} =
t/(1-t)$. The fact that the coefficient of $\widetilde W$ vanishes
is due to the shape of the equation satisfied by $D_0(1,w)$ (see
the appendix).

We now insert this expansion into the representation for $D_2$.
Observe that $S$ has an expansion of the form
\[
S = S_{2}\widetilde W^2 + S_3 \widetilde W^3 + \cdots
\]
with $S_2\ne 0$. Thus we have $\sqrt S = \sqrt{S_2} \widetilde W +
O(\widetilde W^2)$. Furthermore, we get expansions for $S_{2,1}$,
$S_{2,2}$, $S_{2,3}$, and $S_{2,4}$. However, we observe that
$S_{2,3}(1,w_3) = 0$ whereas $S_{2,1}(1,w_3) \ne 0$,
$S_{2,2}(1,w_3) \ne 0$, and $S_{2,4}(1,w_3) \ne 0$. Consequently
we can represent $D_2(1,w)$ as
\[
D_2(1,w) = \widetilde D_{2,-1} \frac 1{\widetilde W} + \widetilde
D_{2,0} + \widetilde D_{2,1} \widetilde W +
 + \widetilde D_{22} \widetilde W^2 +  \cdots,
\]
where $\widetilde D_{1,-1} \ne 0$.

In completely the same way it follows that $D_3(1,w)$ has a local expansion
of the form
\[
D_3(1,w) = \widetilde D_{3,-3} \frac 1{\widetilde W^3} +
 + \widetilde D_{3,-1} \frac
1{\widetilde W} + \widetilde D_{3,0} + \widetilde D_{3,1}
\widetilde W + \cdots,
\]
where $\widetilde D_{3,-3} \ne 0$, and the coefficient of
$\widetilde W^{-2}$ vanishes identically.

These type of singular expansions carry over to $B_3(1,w)$, and we
get
\begin{equation}\label{eqB3exp}
B_3(1,w) =
\widetilde B_{3,-1} \frac 1{\widetilde W} + \widetilde B_{3,0} +
\widetilde B_{3,1} \widetilde W + \cdots,
\end{equation}
where $\widetilde B_{3,-1} \ne 0$. We stress the fact that the
coefficients of $\widetilde W^{-3}$ and $\widetilde W^{-2}$ vanish
as a consequence of non trivial cancellations. Hence, we obtain
the proposed asymptotic relation for the $d_k$.
\end{proof}

The following is the analogous of Lemma \ref{sing}. The difference
now is that we are composing two singular expansions and,
moreover, they are of type $3/2$.

\begin{lem}\label{sing-planar}
Let $f(x) = \sum_{n\ge 0} a_n x^n/n!$ denote the exponential
generating function of a sequence $a_n$ of non-negative real
numbers and suppose that $f(x)$ has exactly one dominating
singularity at $x=\rho$ of the form
$$
f(x) = f_0 + f_2X^2 + f_3X^3 + \mathcal{O}(X^4),
$$
where $X = \sqrt{1-x/\rho}$, and has an analytic continuation to
the region $\{ x\in \mathbb{C} : |x|< \rho+\varepsilon\} \setminus
\{ x \in \mathbb{R} : x \ge \rho\}$ for some $\varepsilon > 0$.
Further, let $H(x,z,w)$ denote a function
that has a dominant singularity at $z=f(\rho) > 0$ of the form  
$$
H(x,z,w) = h_0(x,w) + h_2(x,w)Z^2 + h_3(x,w)Z^3 + \mathcal{O}(Z^4),
$$
where $w$ is considered as a parameter, $Z=\sqrt{1-z/f(\rho)}$,
the functions $h_j(x,w)$ are analytic in $x$, and $H(x,z,w)$ has
an analytic continuation in a suitable region.

Then the function
\[
f_H(x) = H(x,f(x),w)
\]
has a power series expansion $f_H(x) = \sum_{n\ge 0} b_n x^n/n!$
and the coefficients $b_n$  satisfy
\begin{equation}\label{sing-eqlimitrelation}
\lim_{n\to\infty} \frac{b_n}{a_n} = - {h_2(\rho,w) \over f_0} + {h_3(\rho,w)
\over f_3} \left( - {f_2 \over f_0} \right)^{3/2}.
\end{equation}
\end{lem}

\begin{proof}
The proof is similar to that of Lemma \ref{eqlimitrelation} and is
based on composing the singular expansion of $H(x,z,w)$  with that
of $f(x)$. Indeed, near $x=\rho$, and taking into account that
$f(\rho) = f_0$,  we have
$$
f_H(x) = h_0(x,w) + h_2(x,w)
 \left( - {f_2X^2 + f_3X^3 \over f_0} \right)
  + h_3(x,w) \left( - {f_2X^2 + f_3X^3 \over f_0} \right)^{3/2} +
  \cdots
$$
Now note that $x = \rho - \rho X^2$. Thus, if we expand
and extract the coefficient of $X^3$ and applying transfer
theorems, we have
$$
{a_n \over n!} \sim {f_3 \over \Gamma(-3/2)} n^{-5/2} \rho^{-n}
$$
and
$$
{b_n \over n!} \sim {1 \over \Gamma(-3/2)} \left(-{h_2(\rho,w) f_3
\over f_0 }+ h_3(\rho,w) \left( - {f_2 \over f_0} \right)^{3/2} \right)
n^{-5/2} \rho^{-n},
$$
so that the result follows.
\end{proof}

\begin{thm}\label{th:planar}
Let $d_k$ be the probability that a vertex of a connected planar
graph has degree $k$. Then
\begin{align*}
p(w) = \sum_{k \ge 1} d_k w^k &=
-e^{B_0(1,w)-B_0(1,1)}B_2(1,w) \\
&+ e^{B_0(1,w)-B_0(1,1)} \frac{1 + B_2(1,1)}{B_3(1,1)} B_3(1,w),
\end{align*}
where  $B_j(y,w)$, $j=0,2,3$, are given in the appendix.

Moreover, $p(1)=1$, so that the $d_k$ are indeed a probability
distribution and we have asymptotically, as $k\to\infty$,
\[
d_k \sim c k^{-1/2} q^k,
\]
where $c \approx 3.0175067$ is a computable constant and $q$ is as
in Theorem \ref{th:twoconnectedplanar}.
\end{thm}

\begin{proof}
The degree distribution is encoded in the function
\[
C^\bullet(x,w) = \sum_{k\ge 1} C_k(x,1) w^k = e^{B^\bullet(xC'(x),1,w)},
\]
where the generating function $xC'(x)$ of connected rooted planar
graphs satisfies the equation
\[
xC'(x)  = xe^{B^\bullet(xC'(x),1,1)}.
\]
From Lemma~\ref{B-planar-sing} we get the local expansions
\[
e^{B^\bullet(x,1,w)} = e^{B_0(1,w)}\left(1 + B_2(1,w) X^2 +
B_3(1,w) X^3 + O(X^4)\right),
\]
where $X = \sqrt{ 1- x/R}$.
Thus, we first get an expansion for $xC'(x)$
\[
xC'(x) = R - \frac{R}{1+B_2(1,1)} \widetilde X^2 + \frac{R
B_3(1,1)} {(1+B_2(1,1))^{5/2}} \widetilde X^3 + O(\widetilde X^4),
\]
where $\widetilde X = \sqrt{ 1- x/\rho}$ and $\rho$ is the radius
of convergence of $C'(x)$ (compare with \cite{gn}). Note also that
$R = \rho e^{B_0(\rho,1)}$. Thus, we can apply
Lemma~\ref{sing-planar} with $H(x,z,w) = x e^{B^\bullet(z,1,w)}$
and $f(x) = xC'(x)$. We have
\[
f_0 = R,\ f_2 = -\frac{R}{1+B_2(1,1)},\ f_3 = \frac{R B_3(1,1)}
{(1+B_2(1,1))^{5/2}}
\]
and
\[
h_0(\rho,w) = \rho e^{B_0(1,w)},\ h_2(\rho,w) = \rho e^{B_0(1,w)}B_2(1,w),\
h_3(\rho,w) = \rho e^{B_0(1,w)}B_3(1,w).
\]

We can express the probability generating function $p(w)$ as
 $$
    \lim_{n\to\infty} { [x^n] x C^\bullet(x,w) \over [x^n]
    xC'(x)}.
 $$
Consequently, we have
\begin{align*}
p(w) &= -{h_2(\rho,w) \over f_0} + {h_3(\rho,w)
\over f_3} \left( - {f_2 \over f_0} \right)^{3/2} \\
&= -e^{B_0(1,w)-B_0(1,1)}B_2(1,w) \\
&+ e^{B_0(1,w)-B_0(1,1)} \frac{1 + B_2(1,1)}{B_3(1,1)} B_3(1,w).
\end{align*}

The singular expansion for $B_2(1,w)$ turns out to be of the form
 $$
 B_2(1,w) =  \widetilde B_{2,0} + \widetilde B_{2,1} \widetilde W +
 \cdots.
 $$
Hence the expansion (\ref{eqB3exp}) for $B_3(1,w)$ gives the
leading part in the asymptotic expansion for $p(w)$. It follows
that $p(w)$ has the same dominant singularity as for 2-connected
graphs and we obtain the asymptotic estimate for the $d_k$ as
claimed, with a different multiplicative constant. This concludes
the proof of the main result.
\end{proof}

\section{Degree distribution according to edge
density}\label{sec:density}

In this section we show that there exists a computable degree
distribution for planar graphs with a given edge density or,
equivalently, given average degree.

In \cite{gn} we showed that, for $\mu$ in the open interval
$(1,3)$,  the number of planar graphs with $n$ vertices and
$\lfloor \mu n \rfloor$ is asymptotically
\begin{equation}\label{eqgnm}
    g_{n,\lfloor \mu n \rfloor}  \sim c(\mu)\, n^{-4} \gamma(\mu)^n
    n!,
\end{equation}
where $c(\mu)$ and $\gamma(\mu)$ are computable analytic functions
of $\mu$. The proof was based on the fact for each $\mu \in
(1,3)$, there exists a value $y = y(\mu)$ such that the generating
function $G(x,y)$ captures the asymptotic behaviour of
$g_{n,\lfloor \mu n \rfloor}$.  The exact equation connecting
$\mu$ and $y$ is
\begin{equation}\label{eq:mu}
        -y {\rho'(y) \over \rho(y)} = \mu,
\end{equation}
where $\rho(y)$ is the radius of convergence $G(x,y)$ as a
function of $x$. More precisely, the idea behind the proof is to
weight a planar graph with $m$ edges by $y^m$. If $g_{nm}$ denotes
the number of planar graphs with $n$ vertices and $m$ edges, then
the bivariate generating function
\[
G(x,y) = \sum_{m,n} g_{nm}\, \frac{x^n}{n!} \,y^m =
\sum_{n\ge 1} g_n(y)\, \frac{x^n}{n!}
\]
can be considered as the generating function of the weighted
numbers $g_n(y) = \sum_m g_{nm} y^m$ of planar graphs of size $n$.
In addition this weighted model induces a modified probability
model. Instead of the uniform distribution the probability of a
planar graph of size $n$ is now given by $y^m/g_n(y)$, where $m$
denotes the number of edges. It follows from the singular
expansion \cite{gn}
\[
G(x,y) = G_0(y) + G_2(y)Y^2 + G_4(y)Y^4 + G_5(y)Y^5 + O(Y^6),
\]
where $Y = \sqrt{1-x/\rho(y)}$, that $g_n(y)$ is
asymptotically given by
\[
g_n(y) \sim g(y)\, n^{-7/2} \rho(y)^{-n} n!
\]
with $g(y) = G_5(y)/\Gamma(-5/2)$. If $y= 1$  we recover the
asymptotic formula for the number of planar graphs of size $n$.
However, since $g_n(y)$ is a power series in $x$ with coefficients
$g_{nm}$ it is possible to compute these numbers by a Cauchy
integral:
\[
g_{nm} = \frac 1{2\pi i} \int_{|y|= r} g_n(y) \, \frac{dy}{y^{m+1}}.
\]
Suppose that $m = \mu n$. Then due to the asymptotic structure of
$g_n(y)$ the essential part of the integrand  behaves like a
power:
\[
g_n(y) y^{-m} \sim g(y) n^{- 7/2} n! \,(y^\mu \rho(y))^{-n} = g(y)
n^{ 7/2} n! \,\exp\left( -n\, \log(y^\mu \rho(y))\right).
\]
Hence, the integral can be approximated with the help of a saddle
point method, where the saddle point equation $\frac
d{dy}\log(y^\mu \rho(y)) = 0$ is precisely  (\ref{eq:mu}). This
leads directly to (\ref{eqgnm}); see \cite{gn} for details.
Informally, one has to append just a saddle point integral at the
very end of the calculations.

We can use exactly the same approach for the degree distribution.
We consider the weighted number of rooted planar graphs of size
$n$, where the root has degree $k$. For fixed $y$, the
corresponding generating function $G^\bullet(x,y,w)$ has the
property that the radius of convergence of $G^\bullet(x,y,w)$ does
not depend on $w$ and is, thus, given by $\rho(y)$. The asymptotic
equivalent for $g_n^\bullet(y,w)$ contains the factor $\rho(y)^n$
and consequently it is possible to read off the coefficient of
$y^m$ with the help of a saddle point method as above.
Note that we have computed in the appendix the coefficients of the
singular expansion of $B(x,y,w)$ as a function not only of $w$ but
also of $y$. Hence, Theorem \ref{th:planar}  extends directly to
the following.

\begin{thm}
Let $\mu \in (1,3)$ and let $d_{\mu,k}$ be the probability that a
vertex of a connected planar graph with edge density $\mu$ has
degree $k$. Let $y$ be the unique positive solution of
(\ref{eq:mu}). Then
\begin{align*}
\sum_{k \ge 1} d_{\mu,k} w^k &=
-e^{B_0(y,w)-B_0(y,1)}B_2(y,w) \\
&+ e^{B_0(y,w)-B_0(y,1)} \frac{1 + B_2(y,1)}{B_3(y,1)} B_3(y,w),
\end{align*}
where  $B_j(y,w)$, $j=0,2,3$ is given in the appendix.
\end{thm}

The probabilities $d_{\mu,k}$ can be computed explicitly using the
expressions for the $B(y,w)$ given in the appendix. As an
illustration in Figure \ref{plotc} we present a cumulative plot
for $k=1,\dots,10$. Each curve gives the probability that a vertex
has degree at most $k$. The abscissa is the value of $\mu \in
(1,3)$, while the ordinate gives the probability. The bottom curve
corresponds to $k=1$ and the top curve to $k\le 10$. The vertical
line is the value $\kappa \approx 2.21$ such that $\kappa n$ is
the asymptotic expected number of edges in planar graphs
\cite{gn}; since the number of edges is strongly concentrated
around $\kappa n$, the probabilities in these abscissa correspond
to the  cumulating values in the third line of Table
\ref{taula-deg}, namely $0.0367284, 0.1993078, 0.4347438,
0.6215175, 0.7510198, 0.8372003,\dots$

\begin{figure}[htb]
\centerline{\includegraphics[width=10cm,height=5cm]{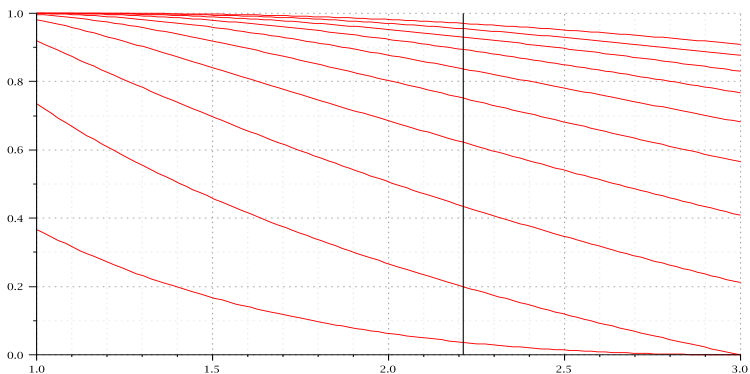}}
\caption{Cumulative degree distribution for connected planar
graphs with  $\mu n$ edges, $\mu \in (1,3)$ and
$k=1,\dots,10$.}\label{plotc}
\end{figure}

Figure \ref{plotb} 
shows the data for 2-connected and 3-connected planar graphs.
Notice that in a 2-connected graph the degrees are at least 2, in
a 3-connected graph they are  at least 3, and a 3-connected graph
has at least $3n/2$ edges. The main abscissa for 2-connected
graphs is equal to $2.26$ (see \cite{bender}), and for 3-connected
graphs is equal to $(7+\sqrt 7)/4 \approx 2.41$ (see Theorem
\ref{th:threeconnectedplanar}).

\begin{figure}[htb]
\centerline{\includegraphics[height=3.2cm]{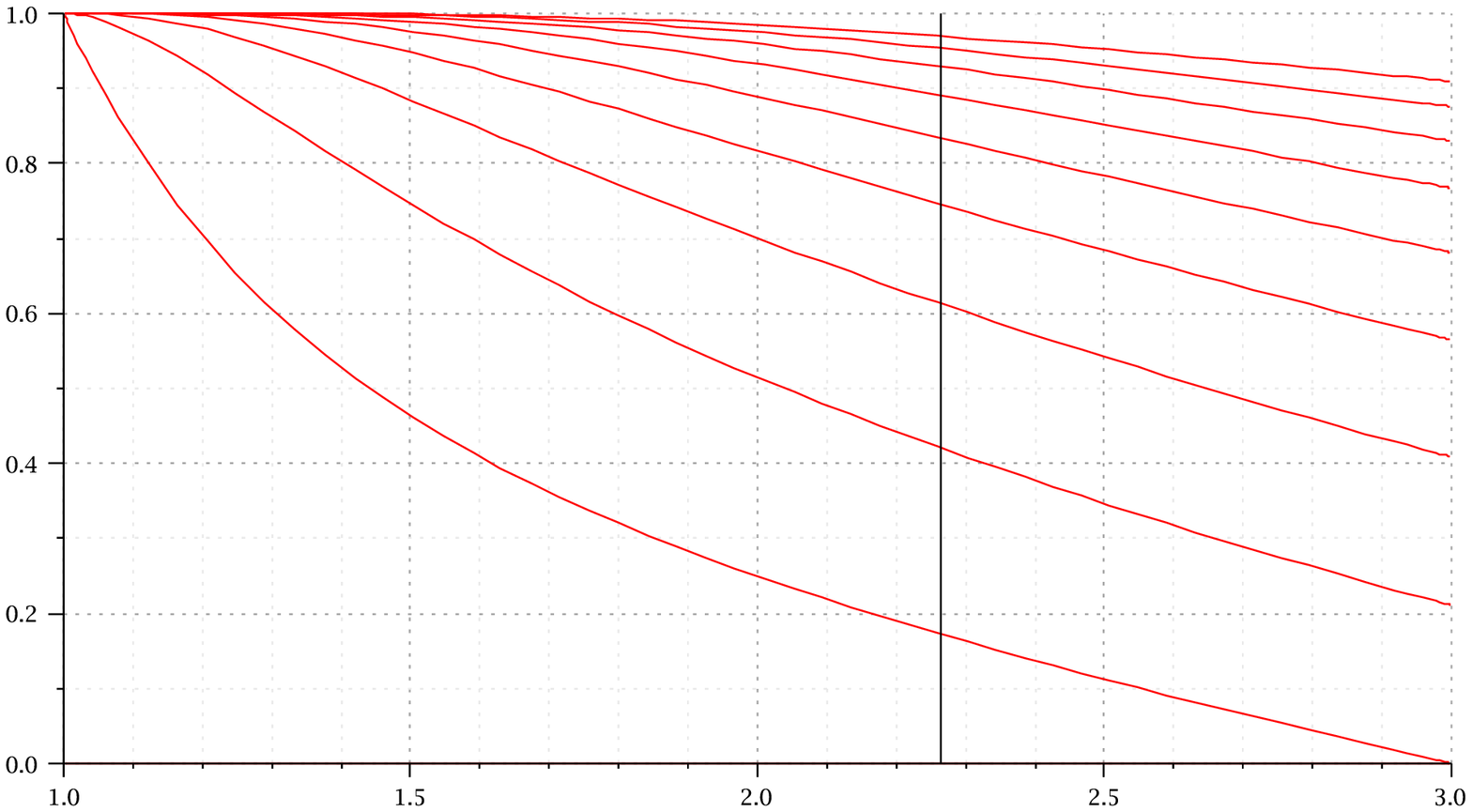}
\hskip1cm
\includegraphics[height=3.5cm]{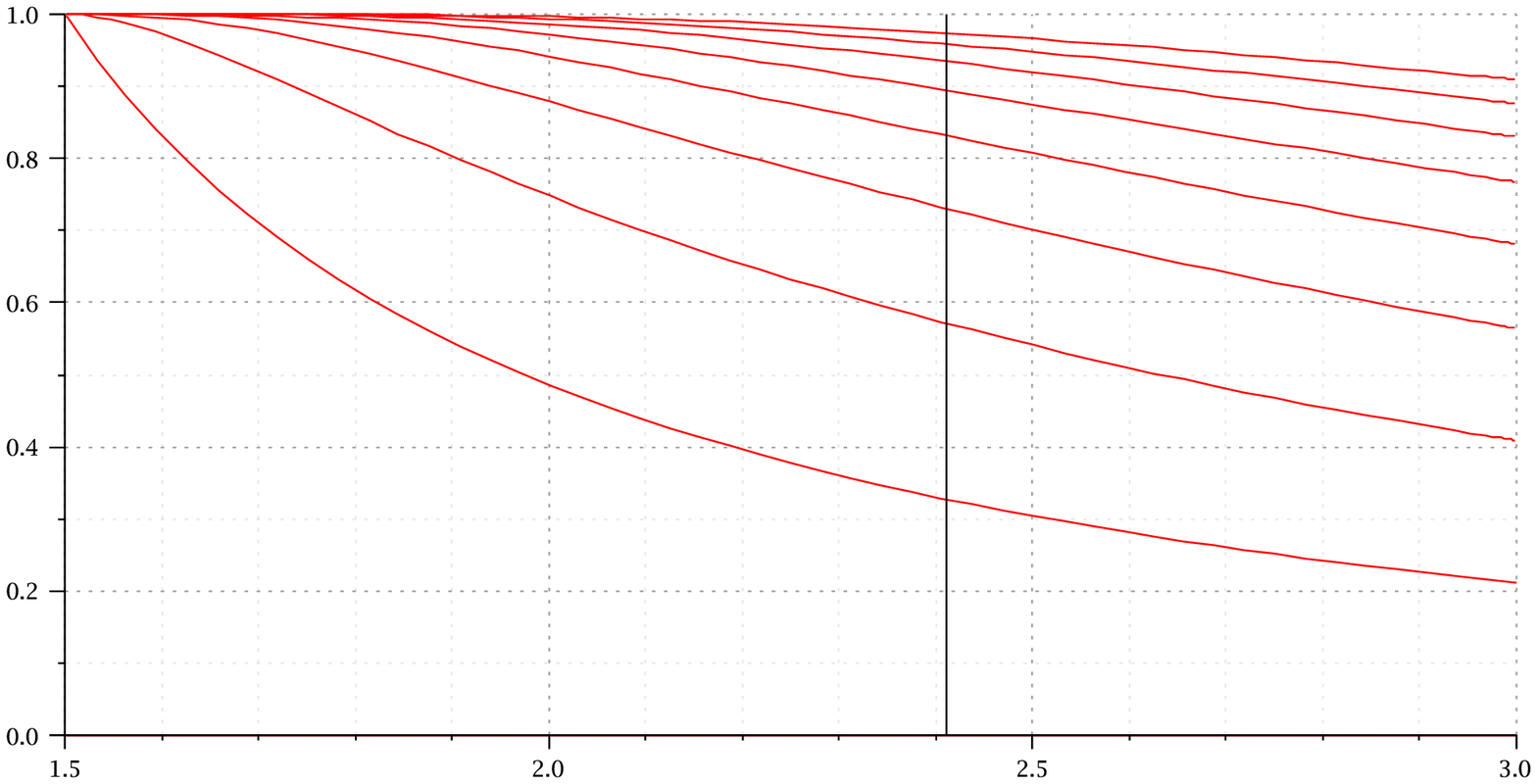}}
 \caption{Cumulative degree distribution for 2-connected planar
graphs with  $\mu n$ edges, $\mu \in (1,3)$ and $k=2,\dots,10$
(left); and 3-connected planar graphs with  $\mu n$ edges, $\mu
\in (3/2,3)$  and $k=3,\dots,10$ (right).}\label{plotb}
\end{figure}


\section{Concluding remarks}\label{sec:conclude}

Fusy \cite{fusy} has designed a very
efficient algorithm for generating random planar graphs uniformly
at random. He has performed extensive experiments on the degree
distribution on planar graphs with 10000 vertices and his
experimental results, which he has very kindly shared with us, fit
very well with the constants in Table \ref{taula-deg}.


\newpage

\section*{Appendix}

\subsection*{The coefficients $T_i(x,w)$}

Let $u$ stand for $u(x,\tau(x))$, which is the solution of
$$x=\frac{(1+u)(3u-1)^3}{16u}.$$
The coefficients $T_i(x,w)$ of the singular expansion of
$T(x,z,w)$ in Lemma~\ref{lem:Ts} are given by

\begin{align*}
&T_0(x,w) =-\frac{(3u-1)^6 w} {27648(3u^2+2u-1+w)(u+1)u^4} \left(
 {P_0\over 9u+1} - (u+1-w)\sqrt{P} \right) \\
& T_2(x,w) = \frac{(3u-1)^6 w} {82944 (3u^2+2u-1+w)^2 (u+1)^2 u^5}
\left( {P_{2,0} \over (9u+1)^2} - {P_{2,1} \over \sqrt P} \right)
\\
& T_3(x,w) = -{(3u-1)^6 w \sqrt{2u(u+1)}  (3u+1) \over 373248
(u+1)^3 u^6 } \left( (3u-1)^2 w  -9u^2-10u-1 + {P_3 \over
(u+1-w)\sqrt P} \right)
\end{align*}
where
$$
P= (u+1-w)(-(3u-1)^2 w+81u^3+99u^2+19u+1)
$$
and
\begin{align*}
 P_0 &=
(27u^2+6u+1)w^2+(-126u^3-150u^2-26u-2)w+81u^4+180u^3+118u^2+20u+1
\\
P_{2,0}& = (1458u^5+3807u^4+900u^3+114u^2-6u-1)w^3+ \\
&(6561u^7+20898u^6+8532u^5-7281u^4-1635u^3-132u^2+30u+3)w^2+ \\
&(-3645u^8-30942u^7-46494u^6-13230u^5+7536u^4+1590u^3-18u^2-42u-3)w
+ \\
&
13122u^9+47385u^8+61560u^7+30708u^6-228u^5-4530u^4-872u^3+36u^2+18u+1
\\
P_{2,1}  &= (-54u^4-45u^3+57u^2-15u+1)w^4+ \\
&(-243u^6+27u^5+1278u^4+858u^3-111u^2+35u-4)w^3+ \\
& (1944u^7+6507u^6+5553u^5-576u^4-1530u^3+15u^2-15u+6)w^2+\\
&(-1215u^8-6561u^7-11439u^6-7005u^5+231u^4+1229u^3+75u^2-15u-4)w
\\
& +
1458u^9+6561u^8+11376u^7+8988u^6+2388u^5-794u^4-512u^3-36u^2+10u+1
\\
P_3 &=-(3u-1)^3w^3+(162u^4+135u^3-27u^2-3u-3)w^2 \\
 &+(81u^5+243u^4+270u^3+138u^2+33u+3)w +
-81u^5-261u^4-298u^3-138u^2-21u-1
\end{align*}

\subsection*{The coefficients $D_i(y,w)$}

We proceed to give the coefficients $D_i(y,w)$ of the singular
expansion of $D(x,y,w)$ in Lemma~\ref{D-planar-sing}. Let
$t=t(y)$, for $y\in(0,\infty)$, be the unique solution in $(0,1)$
of
\[
  y = \frac{(1-2t)}{(1+3t)(1-t)}\exp\left(
-\frac{t^2(1-t)(18+36t+5t^2)}{2(3+t)(1+2t)(1+3t)^2}\right) -1
\]
Then, $D_0(y,w)$ is the solution of

\begin{align*}
1+D_0 &= (1+yw)\exp\left(\frac{\sqrt{S}(D_0(t-1)+t)}{4(3t+1)(D_0+1)}- \right. \\
      & \left. -\frac{D_0^2(t^4-12t^2+20t-9)+D_0(2t^4+6t^3-6t^2+10t-12)+t^4+6t^3+9t^2}{4(t+3)(D_0+1)(3t+1)} \right),
\end{align*}
where $S$ is given by
\[
 S=(D_0t-D_0+t)(D_0(t-1)^3+t(t+3)^2).
\]

The remaining coefficients $D_2(y,w)$, $D_3(y,w)$ are given in
terms of $D_0$ and $t$,

\[
D_2 = \frac{4(D_0+1)^2(t-1)
(S_{2,1}+S_{2,2}\sqrt{S})}{(17t^5+237t^4+1155t^3+2527t^2+1808t+400)(S_{2,3}+S_{2,4}\sqrt{S})}
\]

\begin{align*}
S_{2,1} &= -D_0^2(t-1)^4(t+3)(11t^5+102t^4+411t^3+588t^2+352t+72) \\
       & -D_0t(t-1)(t+3)(22t^7+231t^6+1059t^5+2277t^4+2995t^3+3272t^2+2000t+432) \\
       & -t^2(t+3)^3(11t^5+85t^4+252t^3+108t^2-48t-24) \\
S_{2,2} &= D_0(t-1)(11t^7+124t^6+582t^5+968t^4-977t^3-4828t^2-4112t-984) \\
        & +t(t+3)^2(11t^5+85t^4+252t^3+108t^2-48t-24) \\
S_{2,3} &= -(t+3)(D_0t-D_0+t)(D_0^2(t-1)^4+2D_0(t-1)(t^3-t^2+5t-1)+t(t^3-3t-14) \\
S_{2,4} &=
D_0^2(t^2+2t-9)(t-1)^2+D_0(2t^4-12t^2+80t-6)+t(t^3-3t+50)
\end{align*}

\[
D_3 = \frac{ 24(t+3)(D_0+1)^2 (t-1)t^2(t+1)^2
S_{3,1}^{3/2} \big( S_{3,2}-S_{3,3}(D_0t-D_0+t)\sqrt{S} \big)}{
 \beta^{5/2} (D_0t-D_0+t)\big( S_{3,4} \sqrt{S} - (t+3) (D_0t-D_0+t)
 S_{3,5} \big) }
\]

\begin{align*}
\beta   &= 3t(1+t)(17t^5+237t^4+1155t^3+2527t^2+1808t+400) \\
S_{3,1} & = -5t^5+6t^4+135t^3+664t^2+592t+144 \\
 S_{3,2} & =
  D_0^3(81t^{11}+135t^{10}-828t^9-180t^8+1982t^7+1090t^6-5196t^5\\
  & \qquad \qquad  +2108t^4+2425t^3-1617t^2-256t+256) \\
  &+D_0^2(243t^{11}+1313t^{10}+1681t^9-51t^8-5269t^7-7325t^6+2571t^5 \\
  & \qquad \qquad + 10271t^4+1846t^3-3888t^2-1392t) \\
  &+D_0(243t^{11}+2221t^{10}+8135t^9+15609t^8+12953t^7-3929t^6-12627t^5\\
  & \qquad \qquad -13293t^4-7680t^3-1632t^2) \\
  &+81t^{11}+1043t^{10}+5626t^9+16806t^8+30165t^7+30663t^6+13344t^5+1008t^4-432t^3\\
 S_{3,3} & = D_0(81t^8+378t^7+63t^6-1044t^5+1087t^4-646t^3-687t^2+512t+256) \\
 & +81t^8+800t^7+3226t^6+7128t^5+8781t^4+4320t^3+384t^2-144t \\
  S_{3,4} & =
  D_0^2(t^4-12t^2+20t-9)+D_0(2t^4-12t^2+80t-6)+t^4-3t^2+50t\\
  S_{3,5} & = D_0^2(t^4 - 4t^3 + 6t^2 - 4t +1) + D_0(2t^4 -4t^3
  +12t^2 -12t + 2) + t^4 -3t^2 - 14 t
\end{align*}

\subsection*{The coefficients $B_i(y,w)$}

Consider the singular expansion of the solution to the integral of
Theorem~\ref{integral},
\[
\int_0^w {T^\bullet(x,z,t) \over t}\, dt = I_0(x,w) + I_2(x,w) Z^2
+
 I_3(x,w) Z^3 + O(Z^4)
\]
and define $I_{i,j}(y,w)$ as the coefficients of the singular
expansion of $I_iZ^i$ in terms of $X = \sqrt{1-x/R(y)}$, when
replacing $w$ by $D(x,y,w)/E(x,y,w)$. That is,
\[
\left.I_{i}(x,w)Z^i\right|_{w=D/E} =
I_{i,0}(y,w)+I_{i,2}(y,w)X^2+I_{i,3}(y,w)X^3 + O(X^4)
\]

These coefficients $I_{i,j}$ are given by

\begin{align*}
I_{0,0} &= \frac{(3t+1)(t-1)^3}{2048t^6} \Big(
 4(3t-1)(t+1)^3\log(t)\ +8(3t-1)(t+1)^3\log(t+1) \\
  & +8(3t^4+6t^2-1)\log(2)\ -2(t-1)^3(3t+1)\log(A)\ -2(3t^4+6t^2-1)\log(B) \\
  & +(t-1)(D_0(t^3-3t^2+3t-1)+t^3+4t^2+t+2)\sqrt{S} \\
  & -\frac{t-1}{t+3}\big( D_0^2(t^6-2t^5+t^4-4t^3+11t^2-10t+3) \\
  & +D_0(2t^6+8t^5-10t^4-32t^3+46t^2-8t-6)+t^6+10t^5+34t^4+44t^3+21t^2+18t \big)
 \Big) \\
%
%
I_{0,2} &= \frac{-(3t+1)(t-1)^3}{512t^6}\Big(4(3t^4-4t^3+6t^2-1)\log(2) +
 2(3t^4+6t^2-1)\log(t) + \\
& 4(3t^4+6t^2-1)\log(t+1) + (-3t^4+8t^3-6t^2+1)\log(3A) + (-3t^4-8t^3-6t^2+1)\log(B) \\
& +\frac%
{ D_2(t-1)^5(R_{0,0}\sqrt{S}+R_{0,1})-\frac{(t-1)^2}{\beta(t+3)}(R_{0,4}\sqrt{S}+R_{0,5})}%
{(t+3)\left(R_{0,6}\sqrt{S} + R_{0,3}\right)} \Big) \\
%
%
 I_{0,3} &= -\frac{(3t+1)t^2(t-1)^8(R_{0,0}\sqrt{S}+R_{0,1})(D_3\beta^{5/2}+D_0\alpha^{3/2}R_{0,2})}{512(t+3)t^8\beta^{5/2}\big( R_{0,6}\sqrt{S} + R_{0,3} \big)} \\
%
%
 I_{2,2} &= \frac{(t-1)^6R_{2,0}\left(R_{2,1}\sqrt{S}+R_{2,2}\right)}
 {3072 \beta t^6 (t+1)(D_0+1) } \\
%
%
 I_{2,3} &= \frac{(t+1)^2(1+2t)^2\alpha^{3/2}(3t+1)(t-1)^7\left((t+3)^2R_{2,1}\sqrt{S}-R_{2,2}\right)}{16t^5 \beta^{5/2}(1+D_0)} \\
%
%
 I_{3,3} &= \frac{\sqrt{3}(t-1)^6R_{2,0}^{3/2}\left(R_{3,0}\sqrt{S}-R_{3,1}(D_0t-D_0+t)\right)}{2304\sqrt{3t+1}t^5\beta^{3/2}(t+1)^{3/2}(D_0t-D_0+t)}
\end{align*}

where the expressions $A$, $B$, and polynomials $R_{i,j}$ are
given by

\begin{align*}
A &= D_0(5t^3-3t^2-t-1)+5t^3+6t^2+5t+(3t+1)\sqrt{S} \\
B &= D_0(t^3-3t^2+3t-1)+t^3+2t^2+5t+(t-1)\sqrt{S} \\
R_{0,0} &= 3D_0^2(t-1)^2 - D_0(7t-3) -t(t+3) \\
R_{0,1} &= 3D_0^3(t-1)^4 - D_0^2(t-1)(3t^3-t^2+25t-3) \\
        &  +D_0t(t^3+8t^2+21t-14) +(t+3)^2t^2 \\
R_{0,2} &= 128t(3t+1)(t-1)(1+2t)^2(t+3)^2(t+1)^2 \\
R_{0,3} &= D_0^2(t-1)^4+2D_0(t-1)(t^3+t^2+3t-1)+t^4+4t^3+7t^2+2t+2  \\
R_{0,4} &= 3D_0^3(t-1)^5(51t^8+1081t^7+8422t^6+31914t^5+59639t^4+42461t^3+7584t^2-2832t-864) \\
        &   -D_0^2(t-1)^3(153{t}^{10}+3204{t}^{9}+29055{t}^{8}+146710{t}^{7}+432951{t}^{6} \\
        &   \qquad \qquad +717528{t}^{5}+561457{t}^{4}+208750{t}^{3}+47040{t}^{2}+13248t+2592) \\
        &   -D_0(t+3)^2t(408{t}^{10}+6177{t}^{9}+34003{t}^{8}+92097{t}^{7}+122523{t}^{6}   \\
        &   \qquad \qquad +126075{t}^{5}+145777{t}^{4}+82707{t}^{3}-1543{t}^{2}-15088t-3312) \\
        &   3t(t-1)(t+3)^2(2t^4+3t^3-2t^2+3t+2)(400+1808t+2527t^2+1155t^3+237t^4+17t^5) \\
%
%
\end{align*}
\begin{align*}
R_{0,5} &= 3D_0^4(t-1)^7(51t^8+1081t^7+8422t^6+31914t^5+59639t^4+42461t^3+7584t^2-2832t-864) \\
        &  +2D_0^3(t-1)^4(249t^{10}+3333t^9+22417t^8+105245t^7+339675t^6+\\
        & \qquad \qquad 664087t^5+513315t^4+127943t^3-6936t^2-1152t+1296) \\
        &  -D_0^2t(t-1)^2(357t^{12}+7089t^{11}+58637t^{10}+273500t^9+828314t^8+1886278t^7+3638786t^6\\
        & \qquad \qquad +5441836t^5+4731121t^4+1945329t^3+179665t^2-96240t-20304) \\
        &  -2D_0 t(t+3)^2(51t^{12}+849t^{11}+6580t^{10}+33465t^9+115887t^8+253743t^7+285517t^6 \\
        & \qquad \qquad +148083t^5+130634t^4+141380t^3+59715t^2+4944t-1200) \\
        &  +3t^2(t-1)(t+3)^3(2t^4+3t^3-2t^2+3t+2)(400+1808t+2527t^2+1155t^3+237t^4+17t^5) \\
%
%
R_{0,6} &= D_0(t-1)^2+t^2+t+2 \\
R_{2,0} &= 3(3t+1)(t+1)(-5t^5+6t^4+135t^3+664t^2+592t+144) \\
R_{2,1} &= D_0^2(3t^3-12t^2+7t+2)+D_0(6t^3-3t^2+t)+3t^3+9t^2 \\
R_{2,2} &= D_0^3(3t^7-47t^5-18t^4+21t^3+164t^2-105t-18) \\
 & +D_0^2(9t^7+36t^6-19t^5-168t^4-165t^3+292t^2+15t)\\
 & +D_0(9t^7+72t^6+190t^5+156t^4-63t^3-108t^2)+3t^7+36t^6+162t^5+324t^4+243t^3 \\
R_{3,0} & = D_0^2(t^4-2t^3+2t-1)+D_0(2t^4+4t^3-2t^2-4t)+t^4+6t^3+9t^2 \\
R_{3,1} & =
D_0^2(t^5-3t^4+2t^3+2t^2-3t+1)+D_0(2t^5+6t^4-2t^3-6t^2)+t^5+9t^4+27t^3+27t^2
\end{align*}

Finally, we write the coefficients $B_i(y,w)$ in Lemma
\ref{B-planar-sing} in terms of the  $I_{i,j}$, $D_i$, the radius
of convergence $R=R(y)$, and $y$ and $t$.

\begin{align*}
B_0 &= \frac{1}{128t^3} \Bigl( -8\log(2)(3t^4+6t^2-1)
-8\log(t+1)(3t-1)(t+1)^3
-4\log(t)(3t-1)(t+1)^3  \\
 & +2\log(A)(t-1)^3(3t+1)
 +2\log(B)(3t^4+24t^3+6t^2-1) \\
 & +\sqrt{S}(t-1)(D_0(t^3-3t^2+3t-1)+t^3-8t^2+t-2) \\
 &  -\frac{D_0}{t+3}\bigl(D_0(t-1)^5(t^2+2t-9)+2(t-1)^3)(t^4+60t+3)+(t+3)^2(t-1)(t^3-8t^2+t-2)t\bigr)
 \Bigr) \\
B_2 &=
\frac{RD_0(D_0(R^2E_0^2+RE_2)-2(1+RE_0))}{2(1+RE_0)^2}+\frac{1}{R}(I_{0,0}+I_{0,2}+I_{2,2})\\
    & +\left(\log(1+D_0)-\log(1+yw)-\frac{R^2E_0D_0}{1+RE_0}\right)(1+D_0-D2)R \\
B_3 &= \frac{R^2D_0\left( 2D_3E_0^2R+2D_3E_0+E_3D_0
\right)}{2(E_0R+1)^2}
+RD_3(\log(1+yw)-\log(D_0+1))+\frac{1}{R}(I_{0,3}+I_{2,3}+I_{3,3})
\end{align*}

\end{document}